\documentclass[11pt,a4paper]{article}
\usepackage{a4wide}
\usepackage{latexsym}
\usepackage{amsmath}
\usepackage{amsthm}
\usepackage{amssymb,enumerate}
\usepackage[usenames]{color}
\usepackage{graphicx}
\usepackage{fancybox}
\usepackage{hyperref}
\theoremstyle{plain}
\newtheorem{theo}{Theorem}[section]
\newtheorem{lemma}[theo]{Lemma}
\newtheorem{prop}[theo]{Proposition}
\newtheorem{coro}[theo]{Corollary}
\theoremstyle{definition}
\newtheorem{defi}[theo]{Definition}
\newtheorem{rema}[theo]{Remark}
\newtheorem{exam}[theo]{Example}

\newenvironment{proof1}{\medskip\par\noindent{\bf Proof.}}{\hfill $\Box$
\medskip\par}

\def\C{\mathbb{C}}
\def\N{\mathbb{N}}
\def\R{\mathbb{R}}

\def\a{\alpha}
\def\b{\beta}
\def\o{\omega}
\def\ro{\rho}
\def\ga{\gamma}

\def\eps{\varepsilon}

\def\bM{\mathbb{M}}
\def\M{\mathbb{M}}
\def\L{\mathbb{L}}
\def\bm{\boldsymbol{m}}
\def\m{\boldsymbol{m}}

\def\en{\subseteq}
\def\parn{\par\noindent}
\newcommand{\pprec}{\prec\mathrel{\mkern-5mu}\prec}

\definecolor{azulosc}{rgb}{0.2,0.1,0.7}
\definecolor{naranjauva}{RGB}{237,110,0}
\definecolor{verdecla}{rgb}{0.1,0.8,0.2} 
\definecolor{verdetlp}{RGB}{33,120,68}
\definecolor{granate}{rgb}{0.6,0,0.3} 
\definecolor{azulclarouva}{RGB}{130,115,186}
\definecolor{moradouva}{RGB}{186,31,181}
\definecolor{moradopod}{RGB}{97,45,98} 
\definecolor{moradomuyclaropod}{RGB}{160,129,161} 
\definecolor{moradoclaropod}{RGB}{149,78,153} 
\definecolor{verdeclaropod}{RGB}{151,194,184} 

\definecolor{verdeuva}{RGB}{122,154,1}
\definecolor{verdeosc}{RGB}{46, 139, 87} 
\definecolor{similar1}{RGB}{88,6,124}
\definecolor{similar2}{RGB}{186,14,0}
\definecolor{opuesto1}{RGB}{0,142,25}
\definecolor{opuesto2}{RGB}{178,185,0}
\definecolor{optotal}{RGB}{109,172,0}
\definecolor{amarillouva}{RGB}{172,104,43}     
\definecolor{rosapalo}{rgb}{0.83,0.48,0.38}   
\definecolor{violeta}{rgb}{0.7,0.5,0.8}       
\definecolor{rojo}{RGB}{219,0,0}                

\makeatletter
\DeclareRobustCommand\widecheck[1]{{\mathpalette\@widecheck{#1}}}
\def\@widecheck#1#2{%
    \setbox\z@\hbox{\m@th$#1#2$}%
    \setbox\tw@\hbox{\m@th$#1%
       \widehat{%
          \vrule\@width\z@\@height\ht\z@
          \vrule\@height\z@\@width\wd\z@}$}%
    \dp\tw@-\ht\z@
    \@tempdima\ht\z@ \advance\@tempdima2\ht\tw@ \divide\@tempdima\thr@@
    \setbox\tw@\hbox{%
       \raise\@tempdima\hbox{\scalebox{1}[-1]{\lower\@tempdima\box
\tw@}}}%
    {\ooalign{\box\tw@ \cr \box\z@}}}
\makeatother

\begin{document}

\title{Injectivity and surjectivity of the asymptotic
Borel map in Carleman ultraholomorphic classes}
\author{Javier Jim\'enez-Garrido, Javier Sanz and Gerhard Schindl}
\date{\today}

\maketitle

{ \small
\begin{center}
{\bf Abstract}
\end{center}

We study the injectivity and surjectivity of the Borel map in three instances: in Roumieu-Carleman ultraholomorphic classes in unbounded sectors of the Riemann surface of the logarithm, and in classes of functions admitting, uniform or nonuniform, asymptotic expansion at the corresponding vertex. These classes are defined in terms of a log-convex sequence $\M$ of positive real numbers. Injectivity had been solved in two of these cases by S. Mandelbrojt and B. Rodr{\'\i}guez-Salinas, respectively, and we completely solve the third one by means of the theory of proximate orders. A growth index $\o(\M)$ turns out to put apart the values of the opening of the sector for which injectivity holds or not. In the case of surjectivity, only some partial results were available by J. Schmets and M. Valdivia and by V. Thilliez, and this last author introduced an index $\ga(\M)$ (generally different from $\o(\M)$) for this problem, whose optimality was not established except for the Gevrey case. We considerably extend here their results, proving that $\ga(\M)$ is indeed optimal in some standard situations (for example, as far as $\M$ is strongly regular) and puts apart the values of the opening of the sector for which surjectivity holds or not.\par
}

\bigskip

\section{Introduction}

In 1886,  H. Poincar\'e boosted the mathematical interest in formal (usually divergent) power series by introducing  the notion of asymptotic expansion, which is a kind of
Taylor expansion which provides successive approximations: a complex function $f$, holomorphic on a sector
$S=\{z\in\C: 0<|z|<r,\ a<\arg(z)<b\}$, admits the complex formal power series $\widehat{f}=\sum_{p=0}^\infty a_p z^p$ as its (uniform) asymptotic expansion at the origin if for every $p\in\N_0=\N\cup\{0\}$ there exists a positive constant $C_p$ such that for every $z\in S$ one has
\begin{equation}
\big|f(z)-\sum_{n=0}^{p-1} a_n z^n\big|\leq C_p|z|^p,
\label{eq.Poincare.asymptotics.intro}
\end{equation}
and we write $f\in\widetilde{\mathcal{A}} (S)$. In this context it is natural to consider the asymptotic Borel map
$\widetilde{\mathcal{B}}: \widetilde{\mathcal{A}} (S)\to \C[[z]]$ sending a function $f$ into its  asymptotic expansion  $\widehat{f}$.\par

In 1916,  J. F.~Ritt showed that this map is surjective for any sector $S$, while it is never injective (given a sector bisected by direction 0, the exponential $\exp(-z^{-\alpha})$, $\alpha>0$, is a nontrivial flat, i.e., asymptotically null, function for a suitable choice of $\alpha$).
Hence, given a formal power series $\widehat{f}$ and a sector $S$, it is in general hopeless to try to assign a well-defined sum to it, in the sense that there is not a unique holomorphic function in $S$ asymptotic to $\widehat{f}$.\par
Crucial and original advances were produced in this sense during the 1970's with the works of  J. P. Ramis~\cite{Ramis1,Ramis2}. He noted that, although the formal power series solutions to differential equations are frequently divergent, under fairly general conditions the rate of growth of their coefficients is not arbitrary. Indeed, a remarkable result of E. Maillet~\cite{Maillet} in 1903 states that for any solution $\widehat{f}=\sum_{p\ge 0}a_pz^p$ of an analytic differential equation there will exist $C,A,k>0$ such that $|a_p|\le CA^p(p!)^{1/k}$ for every $p\in\N_0$.
Inspired by this fact, Ramis introduces and
studies the notion of $k-$summability, that rests on classical results by G. N. Watson and R. Nevannlina and generalizes Borel's summability method. His developments are based on a modification of Poincar\'e's asymptotic expansion where the growth of the constant $C_p$ in~\eqref{eq.Poincare.asymptotics.intro} is made explicit in the form
$C_p=C A^p (p!)^{1/k}$ for some $A,C>0$, what entails the same kind of estimates for the coefficients $a_p$ in $\widehat{f}$. The sequence $\M_{1/k}=(p!^{1/k})_{p\in\N_0}$ is the Gevrey sequence of order $1/k$, $f$ is said to be $1/k-$Gevrey asymptotic to $\widehat{f}$ (denoted by $f\in\widetilde{\mathcal{A}}_{\M_{1/k}}(S)$), and $\widehat{f}$, because of the estimates satisfied by its coefficients, is said to be a $1/k-$Gevrey series ($\widehat{f}\in\C[[z]]_{\M_{1/k}}$). The Borel map, defined in this case from $\widetilde{\mathcal{A}}_{\M_{1/k}}(S)$ to $\C[[z]]_{\M_{1/k}}$, is surjective if and only if the opening of the sector $S$ is smaller than or equal to $\pi/k$ (Borel-Ritt-Gevrey Theorem), and it is injective if and only if the opening is greater than $\pi/k$ (Watson's Lemma). So, in this well-known Gevrey case it turns out that $(0,\infty)$ splits as the disjoint union of the intervals of surjectivity and injectivity.

However, motivated by the study of summability of formal power series solutions to different kind of equations, it is interesting to deal with $\M-$asymptotic expansions, whose estimates in \eqref{eq.Poincare.asymptotics.intro} correspond to a constant $C_p=C A^p M_p$ for some $A,C>0$  and for a suitable sequence $\M=(M_p)_{p\in\N_{0}}$ of positive real numbers.
Such estimates then hold also for the coefficients of the power series involved in \eqref{eq.Poincare.asymptotics.intro}, and the corresponding class of formal power series is denoted by  $\C[[z]]_{\M}$. The main aim of this paper is to widen the knowledge of injectivity and surjectivity results for the Borel map in this general context.

One should emphasize that one may consider three closely related, so-called ultraholomorphic classes of  functions in a sector $S$ of the Riemann surface of the logarithm: the class $\widetilde{\mathcal{A}}^u_\M (S)$ of holomorphic functions with uniform asymptotic expansion in $S$;
the class $\widetilde{\mathcal{A}}_\M (S)$ consisting of holomorphic functions with nonuniform asymptotic expansion in $S$, meaning that \eqref{eq.Poincare.asymptotics.intro} holds for $C_p(T)=C_T A_T^p M_p$ on every proper bounded subsector $T$ of $S$ (instead of uniformly on $S$), where $C_T, A_T>0$ depend on the subsector; and, finally, the class $\mathcal{A}_\M (S)$ of functions
for which there exists $A=A(f)>0$ such that
$$\sup_{z\in S,\, p\in\N_0 }\frac{|f^{(p)}(z)|}{A^p p! M_p}  <\infty.$$
In order to guarantee some stability properties for these classes, and to avoid trivial situations, we will always assume that
$\M$ is a weight sequence, that is, a logarithmically convex sequence such that its sequence of quotients of consecutive terms, $\m=(m_p=M_{p+1}/M_p)_{p\in\N_0}$, tends to infinity. Moreover, since the problems under study do not depend on the bisecting direction of the sector, we will mainly work with sectors $S_\ga$ bisected by the direction $d=0$ and with opening $\pi\ga$.

Injectivity and surjectivity of the Borel map for the corresponding ultradifferentiable classes, consisting of smooth functions on intervals of the real line subject to uniform estimates for their derivatives, have been fully characterized: The Denjoy-Carleman theorem (see, for example, \cite{Hormander1990}) characterizes injectivity or, in other words, the quasianalyticity of the corresponding classes, while the results of H.-J. Petzsche~\cite{Pet} prove that the surjectivity amounts to a so-called `strong nonquasianalyticity' condition for $\M$. As the terminology suggests, the Borel map in this case is never bijective.

Regarding the ultraholomorphic framework, the injectivity for the classes  $\widetilde{\mathcal{A}}^{u}_{\M}(S)$ and $\mathcal{A}_{\M}(S)$  was completely solved, respectively, by S. Mandelbrojt~\cite{Mandelbrojt} and B. Rodr{\'i}guez-Salinas~\cite{Salinas} in the 1950's (see Section~\ref{sect.Inject.Borel.map}), but the rest of the information was far from being complete.\par

The results of S. Mandelbrojt and B. Rodr{\'i}guez-Salinas suggested the introduction of a growth index $\o(\M)$, initially given by the second author~\cite{SanzFlatProxOrder} for strongly regular sequences
(i.e. those logarithmically convex, strongly nonquasianalytic and of moderate growth, see Definition~\ref{defiStrongRegularSequen}), which puts apart the openings of quasianalyticity from those of nonquasianalyticity for the three ultraholomorphic classes considered.
Nevertheless, in general it remained open the question about the quasianalyticity of the class $\widetilde{\mathcal{A}}_{\M}(S_{\o(\M)})$, that is, for sectors of optimal opening $\pi\o(\M)$.\par

A first and partial solution to this situation relies on the concept of proximate order, available since the 1920s and extremely useful in the theory of growth of entire functions, and on some related results of L. S. Maergoiz~\cite{Maergoiz} in 2001: if we define the auxiliary functions $\o_{\M}(t)=\sup_{p\in\N_0} \log(t^p/M_p)$ and $d_{\M}(t):=\log(\o_{\M}(t))/\log(t)$
associated with $\M$, it was shown in~\cite{SanzFlatProxOrder} that, whenever $\M$ is strongly regular and $d_{\M}(t)$ is a nonzero proximate order, one is able to produce nontrivial flat functions in $S_{\o(\M)}$, and a generalized version of Watson's Lemma is available.
Indeed, 
it was observed that, for the previous arguments to work, $d_{\M}$ need not be a nonzero proximate order, but rather be close enough to one such order (we say 
$\M$ admits a nonzero proximate order, see Theorem~\ref{theoAdmissProxOrder}). It is then natural to ask oneself whether every strongly regular sequence admits a nonzero proximate order, and the authors found a negative answer in~\cite[Examples~4.16~and~4.18]{JimenezSanzSchindlLCSNPO}. So, the quasianalyticity of $\widetilde{\mathcal{A}}_{\M}(S_{\o(\M)})$ remained open in some cases.

As said before, for the surjectivity only very partial information was available. After the aforementioned Borel-Ritt-Gevrey Theorem in 1978, and by applying techniques from the ultradifferentiable setting, V.~Thilliez~\cite{Thilliez1} proved in 1995 that for the Gevrey class $\mathcal{A}_{\M_\a}(S_{\gamma})$ one has surjectivity if and only if $\ga<\a$, and gave a linear and continuous extension from $\C[[z]]_{\M_\a,A}$ to $\mathcal{A}_{\M_\a,dA}(S_\ga)$ for every $A>0$, where $d>0$ depends only on $\a$ and $\ga$.  In 2000 J. Schmets and M. Valdivia~\cite{SchmetsValdivia}, by working with some nonclassical ultradifferentiable classes $\mathcal{E}_{r,\M}$, $\mathcal{N}_{r,\M}$ and $\mathcal{L}_{r,\M}$ (see Subsection~\ref{subsectSurjGeneralWeightSequences} for more details), obtained some consequences of surjectivity of the asymptotic Borel map for a general weight sequence $\M$ in the Roumieu and Beurling cases and, in particular, characterized the existence of linear and continuous global extension from $\C[[z]]_{\M}$ to $\mathcal{A}_{\M}(S)$ for any sector $S$ (which is much more demanding than surjectivity) as long as the weight sequence satisfies the property of derivation closedness, namely there exists $A>0$ such that $M_{p+1}\le A^{p+1}M_p$ for every $p\in\N_0$.
In 2003, V. Thilliez~\cite{thilliez} improved their results for strongly regular sequences.
He introduced the index $\ga(\M)$, which for such sequences is always a positive real, and showed that for $0<\ga<\ga(\M)$, $\widetilde{\mathcal{B}}: \mathcal{A}_{\M} (S_\ga)  \to \C[[z]]_{\M}$ is surjective and not injective, and again obtained right inverses for the Borel map with a control on the type appearing in the estimates, see Theorem~\ref{th.ThilliezSurjectivity}.
This theorem was reproved by A. Lastra, S. Malek and the second author~\cite{lastramaleksanz} using the technique of the truncated Laplace transform with a suitable kernel. Finally, in~\cite[Theorem\ 6.1]{SanzFlatProxOrder} the second author generalized the Borel--Ritt--Gevrey theorem for strongly regular sequences such that the auxiliary function $d_{\M}$ is a proximate order (or, less demanding, sequences admitting a nonzero proximate order): the Borel map $\widetilde{\mathcal{B}}: \widetilde{\mathcal{A}}_{\M} (S_\ga)  \to \C[[z]]_{\M}$ is surjective if and only if $0<\ga\le \o(\M)$.

So, some important issues arose:
\begin{itemize}
\item[(i)] First, for the sequences appearing in applications the indices $\ga(\M)$ and $\o(\M)$ always coincide, but only $\ga(\M)\le\o(\M)$ seemed to hold in general. The authors could prove in~\cite[Example~4.18]{JimenezSanzSchindlLCSNPO} that these values may be different even for strongly regular sequences.
\item[(ii)] The admissibility of a nonzero proximate order, which happens to hold for most sequences  appearing in applications, has some important consequences for a weight sequence $\M$: It will be strongly regular and $\ga(\M)=\o(\M)$ (see~\cite[Remark~4.15]{JimenezSanzSchindlLCSNPO}, also~\cite{PhDJimenez}). So, these two different indices were hidden as being just one. Moreover, it is not strange that both indices have appeared in the different statements of Thilliez and the second author related to surjectivity.
\item[(iii)] Since the value of $\o(\M)$ has been shown to be crucial for injectivity, one should decide whether $\ga(\M)$ is really putting apart the values of surjectivity from those of nonsurjectivity, and so Thilliez's result is optimal in this sense.
\end{itemize}

After Section~\ref{sectPrelimin}, dedicated to the necessary preliminaries, Section~\ref{sect.Inject.Borel.map} is devoted to solving the injectivity problem. Our first important result in this paper, Theorem~\ref{teorconstrfuncplana}, will show that even the aforementioned assumption of admissibility for $\M$ may be skipped thanks again to the theory of proximate orders and regular variation, concluding that the classes $\widetilde{\mathcal{A}}_{\M}(S_{\o(\M)})$ are always nonquasianalytic.
Moreover, with the help of the quasianalyticity results, we show in Theorem~\ref{theoNotBijectivity} that the Borel map is never bijective, as it ocurred for ultradifferentiable classes.

Our results regarding surjectivity are gathered in Section~\ref{sectSurjectivity}. We start by
showing (Lemma~\ref{lemmaSurjectivityImpliessnq}) that for arbitrary weight sequences, surjectivity for any opening requires $\ga(\M)>0$ or, in other words, $\M$ has to be strongly nonquasianalytic. Without any other assumption on $\M$, no result stating the surjectivity of the Borel map is available, but we may give some information on the maximal possible opening for which surjectivity could occur by resting on results by Schmets and Valdivia~\cite{SchmetsValdivia} and on the use of suitable Borel-like integral transforms, see Theorems~\ref{th.SurjectivityNonUniformAsymp} and~\ref{th.SurjectivityUniformAsymp.plus.dc} (in the second case, by imposing also (dc), see Table~\ref{tableSurjectivity.lc.snq.dc}).

Finally, in Subsection~\ref{subsectSurjStronglyRegSeq} we concentrate in the case of strongly regular sequences $\M$ and prove, by some ramification arguments, that the results of Thilliez are optimal, in the sense that the index $\ga(\M)$ is really the critical value putting apart the openings of surjectivity from those of nonsurjectivity (although, in some situations, the limiting case $S_{\ga(\M)}$ is still an open problem, see Table~\ref{tableSurjectivityStrongReg}). In Remark~\ref{rema.interval.without.inject.surject} we comment on the implications of the fact that $\ga(\M)<\o(\M)$ concerning the Borel map $\widetilde{\mathcal{B}}$.

We conclude analyzing if the value $\ga(\M)$ belongs to these intervals or not in case the sequence is even better behaved and satisfies, for example, $\ga(\M)=\o(\M)$, or even the stronger condition of admitting a nonzero proximate order (see Table~\ref{tableSurjectAdmitProxOrder}).

\section{Preliminaries}\label{sectPrelimin}
\subsection{Notation}
We set $\N:=\{1,2,...\}$, $\N_{0}:=\N\cup\{0\}$.
$\mathcal{R}$ stands for the Riemann surface of the logarithm, and
$\C[[z]]$ is the space of formal power series in $z$ with complex coefficients.\par\noindent
For $\gamma>0$, we consider unbounded sectors bisected by direction 0,
$$S_{\gamma}:=\{z\in\mathcal{R}:|\hbox{arg}(z)|<\frac{\gamma\,\pi}{2}\}$$
or, in general, bounded or unbounded sectors
$$S(d,\alpha,r):=\{z\in\mathcal{R}:|\hbox{arg}(z)-d|<\frac{\alpha\,\pi}{2},\ |z|<r\},\quad
S(d,\alpha):=\{z\in\mathcal{R}:|\hbox{arg}(z)-d|<\frac{\alpha\,\pi}{2}\}$$
with bisecting direction $d\in\R$, opening $\alpha\,\pi$ and (in the first case) radius $r\in(0,\infty)$.\par\noindent
A sectorial region $G(d,\a)$ with bisecting direction $d\in\R$ and opening $\alpha\,\pi$ will be a connected open set in $\mathcal{R}$ such that $G(d,\a)\subset S(d,\a)$, and
for every $\beta\in(0,\a)$ there exists $\rho=\rho(\beta)>0$ with $S(d,\beta,\rho)\subset G(d,\a)$. We simply write $G_{\a}$ for any sectorial region bisected by direction $d=0$ and opening $\alpha\,\pi$. In particular, sectors are sectorial regions.\par\noindent
A bounded (respectively, unbounded) sector $T$ is said to be a \emph{proper subsector} of a sectorial region $G$ (resp. of an unbounded sector $S$) , and we write $T\ll G$ (resp. $T\pprec S$), if $\overline{T}\subset G$ (where the closure of $T$ is taken in $\mathcal{R}$, and so the vertex of the sector is not under consideration).

\subsection{Sequences and associated functions}\label{subsectstrregseq}

In what follows, $\bM=(M_p)_{p\in\N_0}$ will always stand for a sequence of
positive real numbers, and we will always assume that $M_0=1$. The following properties for such a sequence will play a role in this paper.

\begin{defi}\label{defiStrongRegularSequen} We say that:
\begin{itemize}
\item[(i)]  $\M$ is \emph{logarithmically convex} (for short, (lc)) if
$$M_{p}^{2}\le M_{p-1}M_{p+1},\qquad p\in\N.$$%
\item[(ii)] $\M$ is \emph{stable under differential operators} or satisfies the \emph{derivation closedness condition} (briefly, (dc)) if there exists $D>0$ such that $$M_{p+1}\leq D^{p+1} M_{p},  \qquad p\in\N_{0}. $$
\item[(iii)]  $\M$ is of, or has, \emph{moderate growth} (briefly, (mg)) whenever there exists $A>0$ such that
$$M_{p+q}\le A^{p+q}M_{p}M_{q},\qquad p,q\in\N_0.$$
\item[(iv)] $\M$ is \emph{nonquasianalytic} (for short, (nq)) if $$\sum^\infty_{k= 0} \frac{M_k}{(k+1)M_{k+1}}<\infty.$$
\item[(v)]  $\M$ satisfies the \emph{strong nonquasianalyticity condition} (for short, (snq)) if there exists $B>0$ such that
$$
\sum^\infty_{q= p}\frac{M_{q}}{(q+1)M_{q+1}}\le B\frac{M_{p}}{M_{p+1}},\qquad p\in\N_0.$$
\end{itemize}
According to V.~Thilliez~\cite{thilliez}, if $\M$ is (lc), has (mg) and satisfies (snq), we say that $\M$ is \emph{strongly regular}.
\end{defi}

Obviously, (mg) implies (dc), and (snq) implies (nq).

\begin{defi}
For a sequence $\M$ we define {\it the sequence of quotients}  $\m=(m_p)_{p\in\N_0}$ by
$$m_p:=\frac{M_{p+1}}{M_p} \qquad p\in \N_0.$$%
\end{defi}

\begin{rema}\label{remaConseqLC}
The sequence of quotients $\bm$ is nondecreasing if and only if $\M$ is (lc). In this case, it is well-known that $(M_p)^{1/p}\leq m_{p-1}$ for every $p\in\N$, the sequence $((M_p)^{1/p})_{p\in\N}$ is nondecreasing, and $\lim_{p\to\infty} (M_p)^{1/p}= \infty$ if and only if $\lim_{p\to\infty} m_p= \infty$.

We will restrict from now on to
(lc) sequences $\M$ such that $\displaystyle\lim_{p\to\infty} m_p =\infty$, which will be called \emph{weight sequences}
(the last assumption is included in order to avoid trivial situations, see for example Remark~\ref{remaTrivialWhenmnotInfty}). It is immediate that if $\M$ is (lc) and (snq), then $\M$ is a weight sequence.
\end{rema}

\begin{exam}\label{exampleSequences}
We mention some interesting examples. In particular, those in (i) and (iii) appear in the applications of summability theory to the study of formal power series solutions for different kinds of equations.
\begin{itemize}
\item[(i)] The sequences $\M_{\a,\b}:=\big(p!^{\a}\prod_{m=0}^p\log^{\b}(e+m)\big)_{p\in\N_0}$, where $\a>0$ and $\b\in\R$, are strongly regular (in case $\b<0$, the first terms of the sequence have to be suitably modified in order to ensure (lc)). In case $\b=0$, we have the best known example of strongly regular sequence, $\M_{\a}:=\M_{\a,0}=(p!^{\a})_{p\in\N_{0}}$, called the \emph{Gevrey sequence of order $\a$}.
\item[(ii)] The sequence $\M_{0,\b}:=(\prod_{m=0}^p\log^{\b}(e+m))_{p\in\N_0}$, with $\b>0$, is (lc), (mg) and $\bm$ tends to infinity, but (snq) is not satisfied.
\item[(iii)] For $q>1$, $\M_q:=(q^{p^2})_{p\in\N_0}$ is (lc) and (snq), but not (mg).
\end{itemize}
\end{exam}

For weight sequences, the auxiliary functions $\o_{\M}(t)$ and $h_{\M}(t)$, already appearing in the works of S.~Mandelbrojt~\cite{Mandelbrojt}, H.~Komatsu~\cite{komatsu} or V. Thilliez~\cite{thilliez},  play an important role. The map $h_{\bM}:[0,\infty)\to\R$ is defined by
\begin{equation*}
h_{\bM}(t):=\inf_{p\in\N_{0}}M_{p}t^p,\quad t>0;\qquad h_{\bM}(0)=0,
\end{equation*}
and it turns out to be a nondecreasing continuous map in $[0,\infty)$ onto $[0,1]$. In fact
$$
h_{\bM}(t)= \begin{cases}
  t^{p}M_{p}&\quad\textrm{if }t\in\big[\frac{1}{m_{p}},\frac{1}{m_{p-1}}\big),\ p=1,2,\ldots,\\
  1&\quad\textrm{if }t\ge 1/m_{0}.
\end{cases}
$$
One may also consider the function
\begin{equation*}
\o_{\M}(t):=\sup_{p\in\N_{0}}\log\big(\frac{t^p}{M_{p}}\big)= -\log\big(h_{\bM}(1/t)\big),\quad t>0;\qquad \o_{\M}(0)=0,
\end{equation*}
which is a nondecreasing continuous map in $[0,\infty)$ with $\lim_{t\to\infty}\o_{\M}(t)=\infty$.
Indeed,
\begin{equation*}
\o_{\M}(t)=\begin{cases} p\log t -\log(M_{p}) &\quad\textrm{if }t\in [m_{p-1},m_{p}),\ p=1,2,\ldots,\\
0 &\quad\textrm{if } t\in [0,m_{0}).
\end{cases}
\end{equation*}
\parn

\begin{defi}[\cite{Pet},~\cite{chaucho}]
Two sequences $\bM=(M_{p})_{p\in\N_0}$ and $\bM'=(M'_{p})_{p\in\N_0}$ of positive real numbers are said to be \emph{equivalent}, and we write $\M\approx\M'$, if there exist positive constants $L,H$ such that
$$L^pM_p\le M'_p\le H^pM_p,\qquad p\in\N_0.$$
\end{defi}
In this case, it is straightforward to check that
\begin{equation*}
h_{\bM}(Lt)\le h_{\bM'}(t)\le h_{\bM}(Ht),\qquad t\ge 0.
\end{equation*}

\subsection{Asymptotic expansions, ultraholomorphic classes and the asymptotic Borel map}\label{subsectCarlemanclasses}

In this paragraph $G$ is a sectorial region and $\M$ a  sequence. We start recalling the concept of asymptotic expansion.

We say a holomorphic function $f$ in $G$ admits the formal power series $\widehat{f}=\sum_{p=0}^{\infty}a_{p}z^{p}\in\C[[z]]$ as its $\bM-$\emph{asymptotic expansion} in $G$ (when the variable tends to 0) if for every $T\ll G$ there exist $C_T,A_T>0$ such that for every $p\in\N_0$, one has
\begin{equation*}\Big|f(z)-\sum_{n=0}^{p-1}a_nz^n \Big|\le C_TA_T^pM_{p}|z|^p,\qquad z\in T.
\end{equation*}
We will write $f\sim_{\bM}\widehat{f}$ in $G$. $\widetilde{\mathcal{A}}_{\M}(G)$ stands for the space of functions admitting $\bM-$asymptotic expansion in $G$.\par

We say a holomorphic function $f:G\to\C$ admits $\widehat{f}$ as its \emph{uniform $\bM-$asymptotic expansion in $G$ (of type $1/A$ for some $A>0$)} if there exists $C>0$ such that for every $p\in\N_0$, one has
\begin{equation}\Big|f(z)-\sum_{n=0}^{p-1}a_nz^n \Big|\le CA^pM_{p}|z|^p,\qquad z\in G.\label{desarasintunifo}
\end{equation}
In this case we write $f\sim_{\bM}^u\widehat{f}$ in $G$, and $\widetilde{\mathcal{A}}^u_{\M}(G)$ denotes the space of functions admitting uniform $\bM-$asymptotic expansion in $G$. Note that, taking $p=0$ in~\eqref{desarasintunifo}, we deduce that every function in $\widetilde{\mathcal{A}}^u_{\M}(G)$ is a bounded function.

Finally, we define for every $A>0$ the class $\mathcal{A}_{\M,A}(G)$ consisting of the functions holomorphic in $G$ such that
$$
\left\|f\right\|_{\M,A}:=\sup_{z\in G,n\in\N_{0}}\frac{|f^{(n)}(z)|}{A^{n}n!M_{n}}<\infty.
$$
($\mathcal{A}_{\M,A}(G),\left\| \ \right\| _{\M,A}$) is a Banach space, and $\mathcal{A}_{\M}(G):=\cup_{A>0}\mathcal{A}_{\M,A}(G)$ is called a \emph{Roumieu-Carleman ultraholomorphic class} in the sectorial region $G$.

\begin{rema}\label{rema.alg.prop.ultra.class} For any sequence $\M$, the classes $\mathcal{A}_{\M}(G)$, $\widetilde{\mathcal{A}}^u_{\M}(G)$ and $\widetilde{\mathcal{A}}_{\M}(G)$ are complex vector spaces. If $\M$ is (lc), they are algebras and if $\M$ is (dc), they are stable under taking derivatives.\par\noindent
Moreover, if $\M\approx\L$ the corresponding classes coincide.
\end{rema}

For a sector $S$, since the derivatives of $f\in\mathcal{A}_{\bM,A}(S)$ are Lipschitzian, for every $n\in\N_{0}$ one may define
\begin{equation}\label{eq.deriv.at.0.def}
f^{(p)}(0):=\lim_{z\in S,z\to0 }f^{(p)}(z)\in\C.
\end{equation}

As a consequence of Taylor's formula and Cauchy's integral formula for the derivatives, there is a close relation between Roumieu-Carleman ultraholomorphic classes and the concept of asymptotic expansion (the proof may be easily adapted from~\cite{balserutx,galindosanz}).

\begin{prop}\label{propcotaderidesaasin}
Let $\M$ be a sequence, $S$ a sector and $G$ a sectorial region. Then,
\begin{enumerate}[(i)]
\item If $f\in\mathcal{A}_{\M,A}(S)$ then $f$ admits $\widehat{f}:=\sum_{p\in\N_0}\frac{1}{p!}f^{(p)}(0)z^p$ as its uniform $\bM-$asymptotic expansion in $S$ of type $1/A$, where $(f^{(p)}(0))_{p\in\N_0}$ is given by \eqref{eq.deriv.at.0.def}. Consequently, we have that
$$\mathcal{A}_{\M}(S)\subseteq \widetilde{\mathcal{A}}^u_{\M}(S) \subseteq  \widetilde{\mathcal{A}}_{\M}(S).$$
\item $f\in\widetilde{\mathcal{A}}_{\M}(G)$ if and only if for every $T\ll G$ there exists $A_T>0$ such that $f|_T\in \mathcal{A}_{\M,A_T}(T)$.
In case any of the previous holds and $f\sim_\M\sum^\infty_{p=0} a_p z^p$, then for every $T\ll G$ and every $p\in\N_0$ one has
\begin{equation}\label{equaCoeffAsympExpanLimitDeriv}
a_p=\lim_{ \genfrac{}{}{0pt}{}{z\to0}{z\in T}} \frac{f^{(p)}(z)}{p!},
\end{equation}
and we can set ${f^{(p)}(0)}:=p!a_p$.
\item If $S$ is unbounded and $T\pprec S$, then there exists a constant $c=c(T,S)>0$ such that the restriction to $T$, $f|_T$, of functions $f$ defined on $S$ and admitting uniform $\bM-$asymptotic expansion in $S$ of type $1/A>0$, belongs to $\mathcal{A}_{\bM,cA}(T)$.
\item If $f\in\widetilde{\mathcal{A}}_{\M}(G)$, its $\M-$asymptotic expansion $\widehat{f}$ is unique.
\end{enumerate}
\end{prop}

One may accordingly define classes of formal power series
$$\C[[z]]_{\M,A}=\Big\{\widehat{f}=\sum_{n=0}^\infty a_nz^n\in\C[[z]]:\, \left|\,\boldsymbol{a} \,\right|_{\M,A}:=\sup_{p\in\N_{0}}\displaystyle \frac{|a_{p}|}{A^{p}M_{p}}<\infty\Big\}.$$%
$(\C[[z]]_{\M,A},\left| \  \right|_{\M,A})$ is a Banach space and we put $\C[[z]]_{\M}:=\cup_{A>0}\C[[z]]_{\M,A}$.

Given $f\in\widetilde{\mathcal{A}}_{\M}(G)$ with $f\sim_{\M}\widehat{f}$, and taking into account~\eqref{equaCoeffAsympExpanLimitDeriv}, it is straightforward that $\widehat{f}\in\C[[z]]_{\M}$, so it is natural to consider the following map.

\begin{defi}
Given a sectorial region $G$, we define the \textit{asymptotic Borel map}
$$
\widetilde{\mathcal{B}}:\widetilde{\mathcal{A}}_{\M}(G)\longrightarrow \C[[z]]_{\M}$$
sending a function $f\in\widetilde{\mathcal{A}}_{\M}(G)$ into its $\M-$asymptotic expansion $\widehat{f}$.
\end{defi}

\begin{rema}\label{rema.image.Borel.map}
If $G$ is a sector $S$, by Proposition~\ref{propcotaderidesaasin}.(i)
we see that the asymptotic Borel map is also well defined on $\mathcal{A}_{\M}(S)$ and $\widetilde{\mathcal{A}}^u_{\M}(S)$.\par\noindent
If $\M$ is (lc), $\widetilde{\mathcal{B}}$ is a homomorphism of algebras; if $\M$ is also (dc), $\widetilde{\mathcal{B}}$ is a homomorphism of differential algebras.
Finally, note that if $\M\approx\L$, then $\C[[z]]_{\M}=\C[[z]]_{\L}$.
\end{rema}

A fundamental role in the discussion  about the injectivity and surjectivity of the asymptotic Borel map will be played by the flat functions.

\begin{defi}
A function $f$ in any of the previous classes is said to be \textit{flat} if $\widetilde{\mathcal{B}}(f)$ is the null power series, in other words, $f\sim_{\M}\widehat{0}$.%
\end{defi}

One may express flatness in $\widetilde{\mathcal{A}}_{\M}(G)$ by means of the associated functions defined in Subsection~\ref{subsectstrregseq}.

 \begin{prop}[\cite{Thilliez2}, Prop.~4]\label{teorcaracfuncplanaAMS}
Given a sequence $\M$, a sectorial region $G$ and a holomorphic function $f$ in $G$, the following are equivalent:
\begin{enumerate}[(i)]
 \item $f\in\widetilde{\mathcal{A}}_{\M}(G)$ and $f$ is flat,
 \item For every bounded proper subsector $T$ of $G$ there exist $c_1,c_2>0$ with
 $$|f(z)|\le c_1e^{-\o_{\M}(1/(c_2|z|))}= c_1 h_{\M} (c_2 |z|),\qquad z\in T. $$
\end{enumerate}
\end{prop}

In the Gevrey case of order $\a$
we recover the classical result that characterizes flatness in terms of exponential decrease of order $1/\a$.

\subsection{Injectivity and surjectivity intervals for the asymptotic Borel map}\label{subsectIntervals}

By using a simple rotation, we see that the injectivity and the surjectivity of the Borel map in any of the previously considered classes do not depend on the bisecting direction $d$ of the sectorial region $G$, so we limit ourselves to the case $d=0$. Moreover, in this paper we will restrict our study to the unbounded sectors $S_{\ga}$, and include comments on what can be said, to our knowledge, for more general sectorial regions. So, we define
\begin{align*}
I_{\M}:=&\{\ga>0; \quad \widetilde{\mathcal{B}}:\mathcal{A}_{\M}(S_\ga)\longrightarrow \C[[z]]_{\M} \text{ is injective}\}  ,\\
\widetilde{I}^u_{\M}:=&\{\ga>0; \quad\widetilde{\mathcal{B}}:\widetilde{\mathcal{A}}^u_{\M}(S_\ga)\longrightarrow \C[[z]]_{\M} \text{ is injective}\},\\
\widetilde{I}_{\M}:=&\{\ga>0;\quad \widetilde{\mathcal{B}}:\widetilde{\mathcal{A}}_{\M}(S_\ga)\longrightarrow \C[[z]]_{\M} \text{ is injective} \}.
\end{align*}
Whenever $\gamma>0$ belongs to any of these sets, we say that the corresponding class is \emph{quasianalytic}. So, nonquasianalyticity amounts to the existence of nontrivial flat functions in the class.

We easily observe that, by restriction and the identity principle, if $\ga>0$ is in any of those sets then every $\ga'>\ga$ also is. Hence, $I_{\M}$, $\widetilde{I}^u_{\M}$ and $\widetilde{I}_{\M}$ are either empty or unbounded intervals contained in $(0,\infty)$, which we call \emph{quasianalyticity or injectivity intervals}.

Similarly, we define
\begin{align*}
S_{\M}:=&\{\ga>0; \quad \widetilde{\mathcal{B}}:\mathcal{A}_{\M}(S_\ga)\longrightarrow \C[[z]]_{\M} \text{ is surjective}\} ,\\
\widetilde{S}^u_{\M}:=&\{\ga>0; \quad\widetilde{\mathcal{B}}:\widetilde{\mathcal{A}}^u_{\M}(S_\ga)\longrightarrow \C[[z]]_{\M} \text{ is surjective}\}, \\
\widetilde{S}_{\M}:=&\{\ga>0;\quad \widetilde{\mathcal{B}}:\widetilde{\mathcal{A}}_{\M}(S_\ga)\longrightarrow \C[[z]]_{\M} \text{ is surjective} \}.
\end{align*}
It is also plain to check that if $\ga>0$ is in any of those sets then every $0<\ga'<\ga$ also is, so $S_{\M}$, $\widetilde{S}^u_{\M}$ and $\widetilde{S}_{\M}$ are either empty or left-open intervals  having $0$ as endpoint, called \emph{surjectivity intervals}.
Using~Proposition~\ref{propcotaderidesaasin}.(i) 
we easily see that
\begin{align}
I_{\M}\supseteq \widetilde{I}^u_{\M} \supseteq\widetilde{I}_{\M},
\label{equaContentionInjectIntervals}\\
S_{\M}\subseteq \widetilde{S}^u_{\M} \subseteq\widetilde{S}_{\M}.
\label{equaContentionSurjectIntervals}
\end{align}

\begin{rema}\label{remaTrivialWhenmnotInfty}
In the forthcoming results we will only deal with weight sequences. The requirement of (lc) condition is motivated in Remarks~\ref{rema.alg.prop.ultra.class} and~\ref{rema.image.Borel.map}. 
In order to justify the limit condition for $\m$, observe that for a (lc) sequence $\M$, if $\lim_{p\to\infty} m_p \not=\infty$ then $\lim_{p\to\infty} m_p <\infty$
and also $\lim_{p\to\infty} (M_p)^{1/p} <\infty$ (see Remark~\ref{remaConseqLC}). Then there exists $A>0$  such that $h_{\M}(t)=0$  for all $t\in[0,A]$. Hence, by Proposition~\ref{teorcaracfuncplanaAMS}, if $G$ is any sectorial region and $f\in\widetilde{\mathcal{A}}_{\M}(G)$ is flat, we have that $f(t)=0$ for every $t\in(0,A]$ which, by the identity principle, implies that  $f(z)$ identically vanishes in $G$. Consequently, the Borel map is always injective.\par

On the other hand, in the same situation, the Borel map is never surjective: Choose $R>0$ such that $R<|z|$ for some $z\in G$. We can consider a holomorphic function at the origin $L(z)$ whose Taylor expansion at $0$ is given by a convergent lacunary series $\widehat{L}\in\C[[z]]_{\M}$, whose domain of convergence is the disc of radius $R$ and has the circle of this radius as its natural boundary. We have that $L\sim_{\M} \widehat{L}$ on a region $G'\subseteq G$, so by the injectivity of the Borel map there cannot exist another function $E\in\widetilde{\mathcal{A}}_{\M}(G)\subseteq \widetilde{\mathcal{A}}_{\M}(G') $  with $E\sim_{\M}\widehat{L}$.
Since $L$ cannot be analytically continued to $G$, the Borel map is not surjective.

\end{rema}

\section{Injectivity intervals: known results, and complete solution of the problem}\label{sect.Inject.Borel.map}

The quasianalyticity intervals $\widetilde{I}^u_{\M}$ and $I_{\M}$ were determined in the literature in the 1950's.
The first case is basically answered by the following result of S. Mandelbrojt in 1952.

\begin{theo}[\cite{Mandelbrojt},\ Section\ 2.4.III]\label{theo.Mandelbrojt}
 Let $\M$ be a weight sequence, $\ga>0$, $b\ge 0$ and
 $$
 H_b=\{z\in\C: \Re(z)>b \}.
 $$
 The following statements are equivalent:
 \begin{enumerate}[(i)]
  \item $\displaystyle \sum_{p=0}^{\infty} \left(\frac{1}{m_p} \right)^{1/\ga}$ diverges,
  \item If $f$ is holomorphic in $H_b$ and there exist $A,C>0$ such that
  \begin{equation}\label{eq.BoundsTh.Mandel}
  |f(z)|\leq \frac{CA^pM_p}{|z|^{\ga p}}, \qquad z\in H_b, \quad p\in \N_0,
  \end{equation}
  then $f$ identically vanishes.
 \end{enumerate}
\end{theo}

On the one hand, observe that a function $f$ is holomorphic in $H_0$ and verifies the estimates \eqref{eq.BoundsTh.Mandel} if and only if the function $g$ given by $g(z):=f(1/z^{1/\ga})$ belongs to $\widetilde{\mathcal{A}}_{\M}^u(S_{\ga})$ and is flat.

On the other hand, the study of the divergence of the series in (i) is governed by the so-called exponent of convergence of the sequence $\bm$, appearing in the classical theory of growth and factorization of entire functions.

\begin{prop}[\cite{holland}, p.\ 65]
Let $(c_p)_{p\in\N_0}$ be a nondecreasing sequence of positive real numbers tending to infinity. The \emph{exponent of convergence} of $(c_p)_p$ is defined as
$$
\lambda_{(c_p)}:=\inf\{\mu>0:\sum_{p=0}^\infty \frac{1}{c_p^{\mu}}\textrm{ converges}\}
$$
(if the previous set is empty, we put $\lambda_{(c_p)}=\infty$). Then, one has
\begin{equation*}
\lambda_{(c_p)}=\limsup_{p\to\infty}\frac{\log(p)}{\log(c_p)}.
\end{equation*}
\end{prop}

We consider now the closely related growth index (introduced in~\cite{SanzFlatProxOrder}, see also~\cite{JimenezSanzSRSPO}) for weight sequences $\M$,
$$\o(\M):= \displaystyle\liminf_{p\to\infty} \frac{\log(m_{p})}{\log(p)}\in[0,\infty],
$$
and we easily see that
\begin{equation}\label{equaOmegaMExpoConver}
\o(\M)=\frac{1}{\lambda_{(m_p)}}=\frac{1}{\lambda_{((p+1)m_p)}} -1,
\end{equation}
or, in other words,
$$ \o(\M)=\sup\{\mu>0: \sum^\infty_{p=0} \frac{1}{(m_p)^{1/\mu}}<\infty \},$$
\begin{equation}\label{equaOmegaExponConvergenceWidehatM}
\o(\M)=\sup\{\mu>0: \sum^\infty_{p=0} \frac{1}{((p+1)m_p)^{1/(\mu+1)}} < \infty \}.
\end{equation}

After all these remarks, we may rephrase Mandelbrojt's result in the following way.

\begin{theo}[\cite{Mandelbrojt}]\label{theoMandelbrojtUniformAsympt}
Let $\M$ be a weight sequence and $\ga>0$. The following statements are equivalent:
\begin{itemize}
\item[(i)] $\widetilde{\mathcal{B}}:\widetilde{\mathcal{A}}^u_{\M}(S_{\gamma}) \longrightarrow \C[[z]] _\M$ is injective.
\item[(ii)] $\sum^{\infty}_{p=0} (m_p)^{-1/\ga}=\infty$.
\item[(iii)] Either $\ga>\o(\M)$, or $\ga=\o(\M)$ and $\sum^{\infty}_{p=0} (m_p)^{-1/\o(\M)}=\infty$.
\end{itemize}
\end{theo}

Similarly, the knowledge of $\widetilde{I}^u_{\M}$ amounts to the next equivalence $(i)\Leftrightarrow(ii)$ obtained by B.~Rodr{\'\i}guez Salinas~\cite{Salinas} in 1955 (see also~\cite{korenbljum}), whereas the following item $(iii)$ stems again from~\eqref{equaOmegaMExpoConver}.

\begin{theo}[\cite{Salinas}, Thm.\ 12]\label{theoSalinas}
Let $\M$ be a weight sequence and $\ga>0$. The following statements are equivalent:
\begin{itemize}
\item[(i)] $\widetilde{\mathcal{B}}:\mathcal{A}_{\M}(S_{\gamma})\longrightarrow \C[[z]] _\M$ is injective.
\item[(ii)] $\sum^{\infty}_{p=0} ((p+1)m_{p})^{-1/(\ga+1)}=\infty$.
\item[(iii)] Either $\ga>\o(\M)$, or $\ga=\o(\M)$ and $ \sum_{p=0}^{\infty} ((p+1)m_{p})^{-1/(\o(\M)+1)}=\infty$.
\end{itemize}
\end{theo}

From Theorem~\ref{theoMandelbrojtUniformAsympt} one may deduce the following partial generalization of Watson's Lemma for nonuniform asymptotics, included in~\cite{JimenezSanzSRSPO}; although in that paper strongly regular sequences are  mainly considered, the proof given for this result is valid for general weight sequences, so we omit it here.

\begin{theo}[\cite{JimenezSanzSRSPO},\ Theorem\ 2.19]\label{theoPartialGenerWatsonLemma}
Let $\M$ be a weight sequence, $\ga>0$ and $G_\ga$ be any sectorial region of opening $\pi\ga$. The following statements hold:
\begin{itemize}
\item[(i)] If $\ga>\o(\M)$, then $\widetilde{\mathcal{B}}:\widetilde{\mathcal{A}}_{\M}(G_{\gamma}) \longrightarrow \C[[z]] _\M$ is injective.
\item[(ii)] If $\ga<\o(\M)$, then $\widetilde{\mathcal{B}}:\widetilde{\mathcal{A}}_{\M}(G_{\gamma}) \longrightarrow \C[[z]] _\M$ is not injective.
\end{itemize}
\end{theo}

\begin{rema}
For any weight sequence $\M$, the information from the previous results can be summarized as follows:
\begin{enumerate}[(i)]
	\item If $\o(\M)=\infty$, by Theorem~\ref{theoSalinas}, we see  that $I_{\M} =\emptyset$ and~\eqref{equaContentionInjectIntervals} implies  $ I_{\M} = \widetilde{I}^u_{\M}= \widetilde{I}_{\M}=\emptyset$.
	\item If $\o(\M)=0$, by Theorem~\ref{theoPartialGenerWatsonLemma} we observe that $\widetilde{I}_{\M} =(0,\infty)$ and, by~\eqref{equaContentionInjectIntervals},  we have that $ I_{\M} = \widetilde{I}^u_{\M}= \widetilde{I}_{\M}=(0,\infty)$.
	\item If $\o(\M)\in(0,\infty)$, we have the situation described in Table~\ref{tableInjectivity}, where $\sum_{p=0}^{\infty} \sigma_p$ denotes the series $\sum_{p=0}^{\infty} \left((p+1)m_{p}\right)^{-1/(\o(\M)+1)}$  and $\sum_{p=0}^{\infty} \left(m_{p} \right)^{-1/\o(\M)}$ is abbreviated to $\sum_{p=0}^{\infty} \mu_p$ (note that $\sum_{p=0}^{\infty} \sigma_p<\infty$ implies $\sum_{p=0}^{\infty} \mu_p<\infty$ by applying Theorems~\ref{theoMandelbrojtUniformAsympt} and~\ref{theoSalinas} and using that $\mathcal{A}_{\M}(S_{\ga})\subseteq \widetilde{\mathcal{A}}^u_{\M}(S_{\ga})$ ).
 \begin{table}[h]
\centering
\begin{tabular}{l|l|l|l|}
\cline{2-4}

               \rule{0pt}{1.3\normalbaselineskip}
               &  {$\sum_{p=0}^{\infty} \sigma_p=\infty$}
                       & {$\sum_{p=0}^{\infty} \sigma_p=\infty$}
                       & {$\sum_{p=0}^{\infty} \sigma_p<\infty$} \\
                \rule{0pt}{1.3\normalbaselineskip}
                      & {$\sum_{p=0}^{\infty} \mu_p=\infty$}
                       & {$\sum_{p=0}^{\infty} \mu_p<\infty$}
                       & {$\sum_{p=0}^{\infty} \mu_p<\infty$}                                      \\ \hline
 \multicolumn{1}{|l|}{$I_{\M}$}  & $[\omega(\M),\infty)$                        & $[\omega(\M),\infty)$                        & $(\omega(\M),\infty)$                 \rule{0pt}{1.1\normalbaselineskip}         \\ \hline
  \multicolumn{1}{|l|}{$\widetilde{I}^u_{\M}$}& $[\omega(\M),\infty)$                        & $(\omega(\M),\infty)$                        & $(\omega(\M),\infty)$        \rule{0pt}{1.1\normalbaselineskip}               \\ \hline
  \multicolumn{1}{|l|}{$\widetilde{I}_{\M}$}  & $(\omega(\M),\infty)$ or $[\omega(\M),\infty)$? & $(\omega(\M),\infty)$ &  $(\omega(\M),\infty)$  \rule{0pt}{1.1\normalbaselineskip}\\ \hline
\end{tabular}
\caption{Injectivity intervals for a weight sequence with $\o(\M)\in(0,\infty)$.}
\label{tableInjectivity}
\end{table}

\end{enumerate}
In conclusion, we see that the only injectivity interval not determined by the previous results is $\widetilde{I}_{\M}$, and only when $\o(\M)\in(0,\infty)$ and $\sum_{p=0}^{\infty} \left(m_p\right)^{-1/\o(\M)}=\infty$. Indeed, it only rests to decide whether $\o(\M)\in \widetilde{I}_{\M}$ or not. We will show the existence of nontrivial flat functions in the class $\widetilde{\mathcal{A}}_{\M}(S_{\o(\M)})$, and so one always has $\o(\M)\notin \widetilde{I}_{\M}$ and $\widetilde{I}_{\M}=(\o(\M),\infty)$.
\end{rema}

\begin{exam}\label{examsucesalfabetaomegagamma}
 We consider  the sequence $\M_{\a,\b}=\big(p!^{\a}\prod_{m=0}^p\log^{\b}(e+m)\big)_{p\in\N_0}$, $\a>0$, $\b\in\R$, we have that  $\omega(\M_{\a,\b})=\a$. Hence, Table~\ref{table.inject.intervals.for.MAB} contains all the information about the injectivity intervals deduced from the classical results for the sequences $\M_{\a,\b}$.\par

\begin{table}[!htb]
\centering
\begin{tabular}{l|l|l|l|}
\cline{2-4}
               & $\b\leq\a$
               & $ \a<\b\leq\a+1$
               & $\b>\a+1$ \\ \hline
 \multicolumn{1}{|l|}{$I_{\M_{\a,\b}}$}  & $[\a,\infty)$                        & $[\a,\infty)$                        & $(\a,\infty)$                 \rule{0pt}{1.1\normalbaselineskip}         \\ \hline
  \multicolumn{1}{|l|}{$\widetilde{I}^u_{\M_{\a,\b}}$}& $[\a,\infty)$                        & $(\a,\infty)$                        & $(\a,\infty)$        \rule{0pt}{1.1\normalbaselineskip}               \\ \hline
  \multicolumn{1}{|l|}{$\widetilde{I}_{\M_{\a,\b}}$}  &  $(\a,\infty)$ or $[\a,\infty) $? & $(\a,\infty)$ &  $(\a,\infty)$  \rule{0pt}{1.1\normalbaselineskip}\\ \hline
\end{tabular}
\caption{Injectivity intervals for the sequence $\M_{\a,\b}$ with $\a>0$, $\b\in\R$.}
\label{table.inject.intervals.for.MAB}
\end{table}

Note that even if the Gevrey case $\M_{\a}=\big(p!^{\a}\big)_{p\in\N_0}$ belongs to the first column of
Table~\ref{table.inject.intervals.for.MAB}, all the information is known because the function $f(z):=\exp(-1/z^{1/\a}) \sim_{\M_\a} \widehat{0}$ and
 $f \in \mathcal{\widetilde{A}}_{\M_\a}(S_\a)$, so $\widetilde{I}_{\M_{\a}}=(\a,\infty)$. As mentioned before, we will find such functions for any sequence $\M$ using proximate orders.

\end{exam}

Watson's Lemma will be proved below for the class $\mathcal{\widetilde{A}}_{\M}$ for arbitrary sectorial regions; regarding the other two classes, the following information is available.

\begin{rema}\label{remaQuasianalUniformAsymptBoundedSectors}
Theorem~\ref{theoMandelbrojtUniformAsympt} holds true for bounded sectors $S(0,\ga,r)$ with similar arguments. If $\sum_{p=0}^{\infty} \left(m_{p} \right)^{-1/\ga}<\infty$ the restriction to $S(0,\ga,r)$ of the nontrivial flat function defined in $S_\ga$ given by Theorem~\ref{theoMandelbrojtUniformAsympt} solves the problem. Hence, we only need to prove (ii)$\Rightarrow$(i).

Consider the transformation $z(w)=1/(w + (1/r)^{1/\ga})^{\ga} $, which maps $H_0$ into a region $D$ contained in $S(0,\ga,r)$. Given a flat function $g\in  \mathcal{\widetilde{A}}^u_{\M}(S(0,\ga,r))$, the function $f(w):=g(z(w))$ is defined in $H_0$ and,
since for every $w\in H_0$ we have $|w+(1/r)^{1/\ga}|>|w|$, we deduce that
$$|f(w)|=|g(z(w))|\leq \frac{CA^p M_p}{|(w+(1/r)^{1/\ga})^\ga|^p}\leq \frac{CA^p M_p}{|w|^{\ga p}}, \qquad w\in H_0, \quad p\in \N_0,$$
for suitable $C,A>0$.
By Mandelbrojt's theorem~\ref{theo.Mandelbrojt}, $f$ identically vanishes, and so does $g$.\par\noindent
For more general regions, including sectorial regions, the solution was also given by Mandelbrojt~\cite[Sect.\ 2.4.I]{Mandelbrojt} and the answer depends on the way the boundary of the region approaches the origin.
\end{rema}

\begin{rema}\label{remaTheorSalinasBoundedSectors}
The problem of quasianalyticity for classes of functions with uniformly bounded derivatives in bounded regions has also been treated. In the works of K. V. Trunov and R. S. Yulmukhametov~\cite{TrunovYulmukhametov,Yulmukhametov} a characterization is given, for a convex bounded region containing 0 in its boundary, in terms of the sequence $\M$ and also of the way the boundary approaches 0. In particular, for bounded sectors, if $\gamma\le 1$, $d\in\R$ and $r>0$, it turns out that the class $\mathcal{A}_{\M}(S(d,{\gamma},r))$ is quasianalytic precisely when condition (ii) in Theorem~\ref{theoSalinas} is satisfied.
\end{rema}

Now, our aim will be to construct nontrivial flat functions in $\mathcal{\widetilde{A}}_{\M}(S_{\o(\M)})$, what, according to Proposition~\ref{teorcaracfuncplanaAMS}, amounts to obtaining holomorphic functions in $S_{\o(\M)}$ whose growth is suitably controlled by $\o_{\M}(t)$. The notion of proximate order will play a prominent role in this respect.

\begin{defi}[\cite{Valiron}]
We say a real function $\ro(t)$, defined on $(c,\infty)$ for some $c\ge 0$, is a \emph{proximate order} if the following hold:
 \begin{enumerate}[(i)]
  \item $\rho(t)$ is continuous and piecewise continuously differentiable in $(c,\infty)$,
  \item $\ro(t) \geq 0$ for every $t>c$,
  \item $\lim_{t \to \infty} \ro(t)=\ro< \infty$,
  \item $\lim_{t  \to \infty} t \ro'(t) \log(t) = 0$.
 \end{enumerate}
In case the limit $\ro>0$, we say that $\ro(t)$ is a \emph{nonzero proximate order}.
\end{defi}

\begin{exam}\label{examProxOrders} The following are proximate orders:
\begin{itemize}
\item[(i)] $\ro_{\a,\b}(t)=\displaystyle\frac{1}{\a}-\frac{\b}{\a}\frac{\log(\log(t))}{\log(t)}$, $\a>0$, $\b\in\R$.
\item[(ii)] $\rho(t)=\rho+\displaystyle\frac{1}{t^\ga}$ and $\rho(t)=\rho+\displaystyle\frac{1}{\log^\ga(t)}$, $\ro\ge 0$, $\ga>0$.
\end{itemize}
\end{exam}

The next result by L. S. Maergoiz is the key for the construction.

\begin{theo}[\cite{Maergoiz}, Thm.\ 2.4]\label{propanalproxorde}
Let $\ro(t)$ be a nonzero proximate order with $\lim_{t \to \infty} \ro(t)=\ro$. For every $\ga>0$ there exists an analytic function $V(z)$ in $S_\ga$ such that:
  \begin{enumerate}[(i)]
  \item  For every $z \in S_\ga$,
 \begin{equation*}
    \lim_{t \to \infty} \frac{V(zt)}{V(t)}= z^{\ro},
  \end{equation*}
uniformly in the compact sets of $S_\ga$ (i.~e. $V$ is \emph{regularly varying} in $S_{\ga}$).
\item $\overline{V(z)}=V(\overline{z})$ for every $z \in S_\ga$ (where, for $z=(|z|,\arg(z))$, we put $\overline{z}=(|z|,-\arg(z))$).
\item $V(t)$ is positive in $(0,\infty)$, strictly increasing and $\lim_{t\to 0}V(t)=0$.
\item The function $t\in\R\to V(e^t)$ is strictly convex (i.e. $V$ is strictly convex relative to $\log(t)$).
\item The function $\log(V(t))$ is strictly concave in $(0,\infty)$.
\item  The function $\log( V(t))/\log(t)$, $t>0$, is a proximate order and $\displaystyle \lim_{t\to\infty} V(t)/t^{\ro(t)}=1$.
    \end{enumerate}
\end{theo}

We denote by $MF(\ga,\ro(t))$ the class of such functions $V$.
As a consequence of its regular variation, they share a property that will be crucial.

\begin{prop}[\cite{Maergoiz}, Property\ 2.9]\label{propcotaVpartereal}
 Let $\ro(t)$ be a nonzero proximate order with $\lim_{t \to \infty} \ro(t)=\ro>0$, $\ga\ge 2/\ro$ and $V\in MF(\ga, \ro(r))$.
 Then, for every $\a\in(0,1/\ro)$ there exist constants $b>0$ and $R_0>0$ such that
 \begin{equation*}
  \Re(V(z)) \ge b V(|z|), \quad  z\in S_{\a},\ |z|\ge R_0,
 \end{equation*}
 where $\Re$ stands for the real part.
\end{prop}

In~\cite{SanzFlatProxOrder} it was shown how one can construct flat functions in the class $\widetilde{\mathcal{A}}_{\M}(S_{\omega(\M)})$ for  strongly regular sequences such that the auxiliary function $d_{\M}(t):=\log(\o_{\M}(t))/\log(t)$ is a proximate order. In particular, the sequences $\M_{\a,\b}$ satisfy this condition, and so the Table~\ref{table.inject.intervals.for.MAB} can be completed writing $(\a,\infty)$ in its left lower corner. It was also mentioned, see~\cite[Remark\ 4.11]{SanzFlatProxOrder}, that the weaker condition of admissibility of a proximate order (see Theorem~\ref{theoAdmissProxOrder}) is enough.
A better understanding of the connection between proximate orders and sequences has now been achieved, allowing us to extend this last result for arbitrary weight sequences. In fact, the admissibility of a proximate order $\ro(t)$ guarantees that the associated function $\o_\M$ is bounded above and below by a constant times the function $t^{\ro(t)}$. These bounds are needed for most of the results in~\cite{lastramaleksanz2,SanzFlatProxOrder}, but by suitably using the notion of regular variation we will see that the upper bounds are enough for the construction of flat functions.
The existence of a proximate order such that the upper bounds are available is guaranteed for each nonnegative, nondecreasing continuous function of finite upper order by the following classical result.

\begin{theo}[\cite{GoldbergOstrowskii}, Ch.\ 2, Thm.\ 2.1]\label{th.FiniteOrderFuncHasProxOrder}
Let $\o:(a,\infty)\to(0,\infty)$ be a nonnegative, nondecreasing continuous function with $\ro[\o]:=\limsup_{t\to\infty} \log(\o(t))/\log(t)<\infty$. Then, there exists a proximate order $\ro(t)$ with $\lim_{t \to \infty} \ro(t)=\ro[\o]$ such that
\begin{equation*}
\limsup_{t\to\infty} \frac{\o(t)}{t^{\ro(t)}}\in(0,\infty).
\end{equation*}
\end{theo}

We have all the ingredients for the main result in this section.

\begin{theo}\label{teorconstrfuncplana}
Suppose $\M$ is a weight sequence with $\o(\M)\in(0,\infty)$. Then, $\o(\M)$ does not belong to $\widetilde{I}_{\M}$.
\end{theo}

\begin{proof1}
For brevity, put $\o:=\o(\M)$. By Theorem 2.24 in~\cite{SanzLectureNotesBedlewo} (see also Theorem 2.1.30 in~\cite{PhDJimenez}), for the associated function $\o_{\M}$ one has $\ro[\o_{\M}]=1/\o\in(0,\infty)$, and by Theorem~\ref{th.FiniteOrderFuncHasProxOrder} there exists a nonzero proximate order $\ro(t)$, with $\lim_{t\to\infty}\rho(t)=1/\o$, and constants $A_1>0$ and $t_1>0$ such that
\begin{equation}\label{equateorWatsonMhasproxorder}
  \o_{\M}(t)\le A_1t^{\rho(t)},\quad t\ge t_1.
\end{equation}
Take now a function $V\in MF(2\o,\ro(t))$.
The proof will be complete if we show that $G(z):=\exp(-V(1/z))$, which is well defined and holomorphic in the sector $S_{\o}$, belongs to $\widetilde{\mathcal{A}}_{\M}(S_{\o})$ and it is flat, for what we will use Proposition~\ref{teorcaracfuncplanaAMS}. It is enough to work in subsectors $S(0,\b,r_0)\ll S_{\o}$, where $0<\b<\o$ and $r_0>0$. If $z\in S(0,\b,r_0)$, we have $1/z\in S_{\b}$.
On the one hand, according to $(vi)$ in Theorem~\ref{propanalproxorde}, combined with~\eqref{equateorWatsonMhasproxorder}, there exist $A_2>0$ and $t_2>0$ such that
\begin{equation}\label{equateorWatsonMcontrolledV}
  \o_{\M}(t)\le A_2V(t),\quad t\ge t_2.
\end{equation}
On the other hand, Proposition~\ref{propcotaVpartereal} provides us with constants $b>0$ and $R_0>0$ such that
\begin{equation}
  \label{equateorWatsonRealpartV}
  \Re(V(\zeta)) \ge b V(|\zeta|),\qquad \zeta\in S_{\b},\ |\zeta|\ge R_0.
\end{equation}
Choose a positive constant $c$ such that $c>(A_2/b)^{\o}$. By property $(i)$ in Theorem~\ref{propanalproxorde} we have
$$
\lim_{t\to\infty}\frac{V(t/c)}{V(t)}=\left(\frac{1}{c}\right)^{1/\o} <\frac{b}{A_2},
$$
so that there exists $R_1>0$ such that
\begin{equation}\label{equateorWatsonVtcontrolledVtc}
bV(t)> A_2V(t/c),\quad t\ge R_1.
\end{equation}
Let $R_2:=\max(R_0,R_1,ct_2)$ and $r:=R_2^{-1}$. Then, using~\eqref{equateorWatsonRealpartV}, \eqref{equateorWatsonVtcontrolledVtc} and~\eqref{equateorWatsonMcontrolledV}, for $z\in S(0,\b,r)$ we have
$$
-\Re(V(1/z))\le -bV(1/|z|)<-A_2V(1/(c|z|))\le -\o_{\M}(1/(c|z|)),
$$
and so
$$
|G(z)|=e^{-\Re(V(1/z))}\le e^{-\o_{\M}(1/(c|z|))}.
$$
We are done whenever $r\ge r_0$. Otherwise, by compactness there exists $K>0$ such that the inequality
$$
|G(z)|\le Ke^{-\o_{\M}(1/(c|z|))}
$$
is valid throughout $S(0,\b,r_0)$.
\end{proof1}

So, the question mark in Table~\ref{tableInjectivity} can be deleted and the answer for that cell is $(\o(\M),\infty)$, what completes the study of injectivity for unbounded sectors.

Since flat functions in $S_\ga$ provide (by restriction) flat functions in any sectorial region $G_\ga$ of opening $\pi\ga$, Theorems~\ref{theoPartialGenerWatsonLemma} and~\ref{teorconstrfuncplana} imply the following result.

\begin{coro}[Generalized Watson's Lemma for sectorial regions]
  \label{coroGenerWatsonLemmaSectorialRegions}
  Let $\M$ be a weight sequence, $\ga>0$ and $G_\ga$ be a sectorial region. The following statements are equivalent:
\begin{itemize}
\item[(i)] The Borel map $\widetilde{\mathcal{B}}:\widetilde{\mathcal{A}}_{\M}(G_{\gamma}) \longrightarrow \C[[z]] _\M$ is injective.
\item[(ii)] $\ga>\o(\M)$.
\end{itemize}
\end{coro}

We close this section proving that the Borel map is never bijective in this framework.

\begin{theo}\label{theoNotBijectivity}
Let $\M$ be a weight sequence. Then,
$$ S_{\M}\cap I_{\M} = \widetilde{S}^u_{\M}\cap \widetilde{I}^u_{\M}  = \widetilde{S}_{\M}\cap \widetilde{I}_{\M}=\emptyset. $$
In other words, the Borel map is never bijective.
\end{theo}

\begin{proof1}
In all three cases we will show that surjectivity for any $\ga>0$ implies noninjectivity.

(i) Let us see that $\widetilde{S}_{\M}\cap \widetilde{I}_{\M}=\emptyset$. Suppose $\widetilde{\mathcal{B}}:\widetilde{\mathcal{A}}_{\M}(S_\gamma)\longrightarrow \C[[z]]_{\M}$ is surjective. Since it is clear that the series $\sum_{n=0}^{\infty} z^n$ belongs to $\C[[z]]_{\M}$, there exists $f\in\widetilde{\mathcal{A}}_{\M}(S_\gamma)$ such that
 $f(z)\sim_{\M} \sum_{n=0}^{\infty} z^n$. The function $g(z):=f(z)-\sum_{n=0}^{\infty} z^n=f(z)-1/(1-z)$ is holomorphic in $S_\ga\setminus\{1\}$ and,  by the identity principle, cannot vanish identically. Moreover, $g\in \widetilde{\mathcal{A}}_{\M}(S(0,\gamma,1/2))$ and $g(z)\sim_{\M} \widehat{0}$, and so the Borel map is not injective in $\widetilde{\mathcal{A}}_{\M}(S(0,\gamma,1/2))$.
By Corollary~\ref{coroGenerWatsonLemmaSectorialRegions} we see that $\ga\leq \o(\M)$. Again by Corollary~\ref{coroGenerWatsonLemmaSectorialRegions} we conclude that $\widetilde{\mathcal{B}}:\widetilde{\mathcal{A}}_{\M}(S_\gamma)\longrightarrow \C[[z]]_{\M}$ is not injective.

(ii) Let us see that $\widetilde{S}^u_{\M}\cap \widetilde{I}^u_{\M}=\emptyset$. Suppose $\widetilde{\mathcal{B}}:\widetilde{\mathcal{A}}^u_{\M}(S_\gamma)\longrightarrow \C[[z]]_{\M}$ is surjective. Since $z\in\C[[z]]_{\M}$, there exists $f\in\widetilde{\mathcal{A}}^u_{\M}(S_\gamma)$ such that
 $f(z)\sim_{\M} z$ uniformly in $S_\ga$. The function $g(z):=f(z)-z$ is holomorphic in $S_\ga$ and, since $f$ is bounded in $S_\ga$, cannot vanish identically. Furthermore, $g(z)\sim_{\M} \widehat{0}$ uniformly in $S(0,\ga,1)$, so there exist $C,A>0$ such that
for every $z\in S(0,\ga,1)$ one has
$$
|g(z)|
\le CA^pM_p|z|^p,\qquad p\in\N_0.
$$
Hence, the holomorphic function $\psi:\{z\in\C:\Re(z)>0\}\to\C$, defined by $\psi(u)=g(1/u^\ga)$, is not identically
0 and
$$
|\psi(u)|\le\frac{CA^pM_p}{|u|^{\ga p}},\qquad p\in\N_0,\ \Re(u)> 1.
$$
Now, we can apply Theorem~\ref{theo.Mandelbrojt} in $H_1$ and we deduce that $\sum_{n=0}^\infty m_p^{-1/\ga}<\infty$. By Theorem~\ref{theoMandelbrojtUniformAsympt} we conclude that $\widetilde{\mathcal{B}}:\widetilde{\mathcal{A}}^u_{\M}(S_\gamma)\longrightarrow \C[[z]]_{\M}$ is not injective.

(iii) Finally, let us show that $S_{\M}\cap I_{\M} =\emptyset$. If $\widetilde{\mathcal{B}}:\mathcal{A}_{\M}(S_\gamma)\longrightarrow \C[[z]]_{\M}$ is surjective there exists $f\in\mathcal{A}_{\M}(S_\gamma)$ such that $f^{(p)}(0)=\delta_{1,p}$ for every $p\in\N_0$, where $\delta_{1,p}$ is Kronecker's delta. By definition of the class, there exist $C,A>0$ (without loss of generality, we may assume that $C\ge 1$ and $CAM_1\ge 1$) such that
\begin{equation}
  \label{equaEstimatesNonBijectUniformBoundsDeriv}
  |f^{(p)}(z)|\le CA^pp!M_p,\qquad z\in S_{\ga},\ p\in\N_0.
\end{equation}
We consider the Laplace transform of the function $f(z)-z$,
\begin{equation}\label{equaLaplaceTransformSalinas}
g(z):=\int_0^{\infty(\varphi)} e^{-zt}(f(t)-t)\,dt,\qquad z\in S_{\ga+1},
\end{equation}
where the integration is over the half-line parameterized by $r\in(0,\infty)\mapsto re^{i\varphi}$, whose argument is a real number
\begin{equation}\label{equaLaplTranArgumentHalfline}
\varphi\in\left(-\frac{\pi\ga}{2},\frac{\pi\ga}{2}\right)\ \textrm{ such that }\ \arg(z)+\varphi\in\left(-\frac{\pi}{2},\frac{\pi}{2}\right).
\end{equation}
This last condition guarantees the exponential decrease at infinity of the factor $e^{-zt}$ which, together with the linear growth of $f(t)-t$, ascertains that the function $g$ is well defined and holomorphic in $S_{\ga+1}$.
We proceed now to estimate $|g(z)|$. Firstly, parameterizing we have that
\begin{align*}
  |g(z)|&\le \left|\int_0^\infty e^{-re^{i\varphi}z}f(re^{i\varphi})e^{i\varphi}\,dr-
  \int_0^\infty e^{-re^{i\varphi}z}re^{i\varphi}e^{i\varphi}\,dr\right|\\
  &\le \int_0^\infty e^{-r\Re(e^{i\varphi}z)}|f(re^{i\varphi})|\,dr+
  \left| \int_0^\infty e^{-re^{i\varphi}z}r\,dr\right|.
\end{align*}
In the first integral we use~\eqref{equaEstimatesNonBijectUniformBoundsDeriv} for $p=0$ and compute the remaining integral, and in the second one we integrate by parts, and get that
\begin{align}\label{equaLaplTransFirstEstim}
  |g(z)|&\le \frac{C}{\Re(e^{i\varphi}z)}+\left|\frac{1}{e^{i\varphi}z}\int_0^\infty e^{-re^{i\varphi}z}\,dr\right|\nonumber\\
  &\le \frac{C}{\Re(e^{i\varphi}z)}+\frac{1}{|z|\Re(e^{i\varphi}z)}
\end{align}
for every $z\in S_{\ga+1}$. A different estimation is obtained by integration by parts in~\eqref{equaLaplaceTransformSalinas}, taking into account that $f(0)=0$:
\begin{equation}\label{equaLaplTransFirstIntegrParts}
g(z)=\frac{1}{z}\int_0^{\infty(\varphi)} e^{-zt}(f'(t)-1)\,dt, \qquad z\in S_{\ga+1}.
\end{equation}
Now we parameterize and split the integral as before, and use~\eqref{equaEstimatesNonBijectUniformBoundsDeriv} for $p=1$ to obtain that
\begin{equation}\label{equaLaplTransSecondEstim}
|g(z)|\le \frac{CAM_1}{|z|\Re(e^{i\varphi}z)}+\frac{1}{|z|\Re(e^{i\varphi}z)}\le \frac{2CAM_1}{|z|\Re(e^{i\varphi}z)}.
\end{equation}
Finally, if we iterate the integration by parts in~\eqref{equaLaplTransFirstIntegrParts} and use that $f^{(p)}(0)=\delta_{1,p}$, we get for every $p\ge 2$ the identity
\begin{equation*}
g(z)=\frac{1}{z^p}\int_0^{\infty(\varphi)} e^{-zt}f^{(p)}(t)\,dt, \qquad z\in S_{\ga+1}.
\end{equation*}
Using again~\eqref{equaEstimatesNonBijectUniformBoundsDeriv} for $p\ge 2$ we deduce that
\begin{equation}\label{equaLaplTransThirdEstim}
|g(z)|\le \frac{CA^pp!M_p}{|z|^p\Re(e^{i\varphi}z)}.
\end{equation}
Our aim is to apply Theorem~\ref{theo.Mandelbrojt} to the function $h$ given by $h(w)=g(w^{\ga+1})$, $w\in S_1$, when restricted to the half-plane $\{w:\Re(w)>1\}$. Note that the estimates in~\eqref{equaLaplTransFirstEstim} imply for $\Re(w)>1$ (and so $|w|>1$) that
\begin{equation*}
  |h(w)|\le \frac{C}{\Re(e^{i\varphi}w^{\ga+1})}+\frac{1}{|w^{\ga+1}|\Re(e^{i\varphi}w^{\ga+1})}\le \frac{2C}{\Re(e^{i\varphi}w^{\ga+1})}.
\end{equation*}
These last estimates and the ones in~\eqref{equaLaplTransSecondEstim} and~\eqref{equaLaplTransThirdEstim} can now be summed up for $h$ as
\begin{equation*}
  |h(w)|\le  \frac{2CA^p p! M_p}{|w|^{p(\ga+1)}\Re(e^{i\varphi}w^{\ga+1})},\quad \Re(w)>1,\ p\in\N_0.
\end{equation*}
Now we choose $\varphi$ in order to minimize the value $\Re(e^{i\varphi}w^{\ga+1})$. We study two cases:
\begin{itemize}
\item[(i)] If $|\arg(w)|<\ga\pi/(2(\ga+1))$, then $|\arg(w^{\ga+1})|<\ga\pi/2$ and, according to~\eqref{equaLaplTranArgumentHalfline}, we may choose $\varphi=-\arg(w^{\ga+1})$, and we deduce that $\Re(e^{i\varphi}w^{\ga+1})=|w|^{\ga+1}>1$. So, for such $w$ we get
\begin{equation}\label{equaEstim.h.CentralSector}
  |h(w)|\le  \frac{2CA^p p! M_p}{|w|^{p(\ga+1)}},\quad p\in\N_0.
\end{equation}

\item[(ii)] If $|\arg(w)|<[\ga\pi/(2(\ga+1)),\pi/2)$, the previous choice is not possible, and we choose
$$
\varphi_\eps=\begin{cases}
  -\frac{\ga\pi}{2}+\eps&\quad\textrm{if }\arg(w)\in \left(-\frac{\pi}{2},-\frac{\pi\ga}{2(\ga+1)}\right],\\
  \frac{\ga\pi}{2}-\eps&\quad\textrm{if }\arg(w)\in \left[\frac{\pi\ga}{2(\ga+1)},\frac{\pi}{2}\right),
\end{cases}
$$
for any $\eps\in(0,\ga\pi/2)$. So, $\Re(e^{i\varphi_\eps}w^{\ga+1})=|w|^{\ga+1}\cos((\ga+1)|\arg(w)|-\ga\pi/2+\eps)$, and making $\eps$ tend to 0 we obtain that
\begin{equation}\label{equaEstim.h.NonCentralSector}
  |h(w)|\le  \frac{2CA^p p! M_p}{|w|^{p(\ga+1)}|w|^{\ga+1}\cos((\ga+1)|\arg(w)|-\ga\pi/2)},\quad p\in\N_0.
\end{equation}
Now, observe that in this case
$$
0<\frac{\pi}{2}-|\arg(w)|\le (\ga+1)(\frac{\pi}{2}-|\arg(w)|)\le \frac{\pi}{2},
$$
and so
\begin{align*}
|w|\cos\left((\ga+1)|\arg(w)|-\frac{\ga\pi}{2}\right)&=|w|\sin\left((\ga+1)\left(\frac{\pi}{2}-|\arg(w)|\right)\right)\\
&\ge |w|\sin\left(\frac{\pi}{2}-|\arg(w)|\right)=|w|\cos(\arg(w))=\Re(w)>1.
\end{align*}
Since we also have $|w|^{\ga}>1$, from~\eqref{equaEstim.h.NonCentralSector} we obtain the same estimates~\eqref{equaEstim.h.CentralSector} given in the first case.
\end{itemize}
Since $h$ is not identically 0, by Theorem~\ref{theo.Mandelbrojt} we deduce that the series $\sum_{p=0}^\infty ((p+1)m_p)^{-1/(\ga+1)}$ converges, and Theorem~\ref{theoSalinas} implies that $\widetilde{\mathcal{B}}:\mathcal{A}_{\M}(S_{\gamma})\longrightarrow \C[[z]] _\M$ is not injective.
\end{proof1}

\begin{rema}
As an easy consequence we have that if  $\o(\M)<\infty$, then $$S_{\M}\subseteq  \widetilde{S}^u_{\M}\subseteq\widetilde{S}_{\M}\subseteq(0,\o(\M)].$$
\end{rema}

\section{Surjectivity intervals for the Borel map}\label{sectSurjectivity}

In the study of the surjectivity intervals a new index for the sequence $\M$, introduced in this regard by V. Thilliez~\cite[Section\ 1.3]{thilliez}, will play a central role.

\begin{defi}\label{defiindegrowM}
Let $\bM=(M_{p})_{p\in\N_{0}}$ be a strongly regular sequence and $\ga>0$. We say $\bM$ satisfies property $\left(P_{\ga}\right)$  if there exist a sequence of real numbers $m'=(m'_{p})_{p\in\N_0}$ and a constant $a\ge1$ such that: (i) $a^{-1}m_{p}\le m'_{p}\le am_{p}$, $p\in\N$, and (ii) $\left((p+1)^{-\ga}m'_{p}\right)_{p\in\N_0}$ is increasing.

The index $\ga(\bM)$ is then defined as

$$\ga(\bM):=\sup\{\ga\in\R:(P_{\ga})\hbox{ is fulfilled}\}\in(0,\infty).$$
\end{defi}

This definition makes sense for (lc) sequences, and in this case $\ga(\M)\in[0,\infty]$.
Indeed, this index may be equivalently expressed by different conditions:
\begin{enumerate}[(i)]
\item A sequence $(c_p)_{p\in\N_0}$ is \emph{almost increasing} if there exists $a>0$ such that for every $p\in\N_0$ we have that $c_p\leq a c_q $ for every $ q\geq p$.
It was proved in~\cite{JimenezSanzSRSPO} (for strongly regular sequences, but the argument works in general) that for any weight sequence $\M$ one has
$$\ga(\bM)=\sup\{\ga>0:(m_{p}/(p+1)^\ga)_{p\in\N_0}\hbox{ is almost increasing} \}.
$$
\item For any $\b>0$ we say that $\m$ satisfies $(\ga_{\b})$ if there exists $A>0$ such that
\begin{equation*}
(\ga_{\b})\qquad \sum^\infty_{\ell=p} \frac{1}{(m_\ell)^{1/\b}}\leq \frac{A (p+1) }{(m_p)^{1/\b}},  \qquad p\in\N_0.
\end{equation*}
Using this condition, which was introduced for $\beta=1$ by H. Komatsu~\cite{komatsu} (and named $(\ga_{1})$ after H.-J. Petzsche~\cite{Pet}), and generalized for $\b\in\N$ by J. Schmets and M. Valdivia~\cite{SchmetsValdivia}, we can obtain (see~\cite{JimenezSanzSchindlIndex,PhDJimenez}) an alternative expression of the index:
$$ \ga(\M)=\sup\{\b>0; \,\, \m \,\, \text{satisfies} \,\, (\ga_{\b})  \}.$$
\end{enumerate}

In~\cite[Ch.~2]{PhDJimenez} and \cite[Sect.~3]{JimenezSanzSchindlIndex}, the connections between the indices $\ga(\M)$ and $\o(\M)$, the growth properties usually imposed on weight sequences, and the theory of O-regular variation, have been thoroughly studied. In particular, 
whenever $\widehat{\M}=(p!M_p)_{p\in\N_0}$ is (lc) and $\b>0$ we have that
\begin{enumerate}[(i)]
 \item $\ga(\M)>0$ if and only if $\M$ is (snq) (this fact is deduced from the works of K. N.~Bari and S. B.~Ste\v{c}kin~\cite{BariSteckin} and S.~Tikhonov~\cite[Lemma\ 4.5]{Tikhonov}).
 \item $\ga(\widehat{\M})>1$ if and only if $\widehat{\m}$ satisfies $(\ga_1)$.
 \item $\ga(\widehat{\M})>\b$ if and only if $\widehat{\m}$ satisfies $(\ga_\b)$ (this is a consequence of (ii)).
\end{enumerate}

A straightforward verification shows that for every $s>0$ one has
$$
\ga((p!^sM_p)_{p\in\N_0})=\ga(\M)+s,\qquad \ga((M_p^s)_{p\in\N_0})=s\ga(\M),
$$
\begin{equation}\label{equaPropertiesIndexOmega}
\omega((p!^sM_p)_{p\in\N_0})=\omega(\M)+s,\qquad \omega((M_p^s)_{p\in\N_0})=s\omega(\M).
\end{equation}

Next we compare the two indices introduced so far.
\begin{prop}\label{propComparisonIndices}
For any weight sequence $\M$ we always have $\ga(\M)\leq\o(\M)$.
\end{prop}
\begin{proof1}
The statement is trivial if $\ga(\M)=0$. Otherwise, it suffices to prove that whenever $\ga>0$ is such that $(m_{p}/(p+1)^\ga)_{p\in\N_0}$ is almost increasing, one has $\ga\le\o(\M)$. By definition, there exists $a>0$ such that for every $p\in\N_0$ one has $m_0\le am_p/(p+1)^\ga$, and so
$\log(m_p)\ge \ga\log(p+1)+\log(m_0/a)$. From here and by the definition of $\o(\M)$ the conclusion easily follows.
\end{proof1}

\subsection{Weight sequences}\label{subsectSurjGeneralWeightSequences}

Our first result is based on a theorem by H.-J. Petzsche in the ultradifferentiable setting and we need to consider the following space.

\begin{defi}
We say that $f\in\mathcal{E}_{\M}([-1,1])$ if $f\in\mathcal{C}^{\infty}([-1,1])$ and there exists a constant $A>0$ for which
$$
\sup_{p\in\N_0,\, x\in[-1,1]}\frac{|f^{(p)}(x)|}{A^pp!M_p}<\infty.
$$
\end{defi}

Correspondingly, we consider the Borel map $\mathcal{B}:\mathcal{E}_{\M}([-1,1])\longrightarrow \C[[z]]_{\M}$ sending $f$ into the formal power series $\sum^\infty_{p=0} (f^{(p)}(0)/p!) z^p$ (we warn the reader our notations differ from those in~\cite{Pet}).

All over the paper~\cite{Pet}, H.-J. Petszche assumes that $\widehat{\M}$ is a weight sequence and that $\M$ satisfies (nq). However, condition (nq) can be suppressed in the statement of the following theorem, since, if $\widehat{\m}=((p+1)m_p)_{p\in\N_0}$ satisfies ($\ga_1$) then $\M$ satisfies (snq) and, consequently, (nq), and there is only one direction that needs to be checked. This can be done by carefully inspecting his proof.

\begin{theo}[\cite{Pet}, Thm.\ 3.5]\label{th.Petszche}
Let $\M$ be a sequence such that $\widehat{\M}$ is weight sequence. Then, the Borel map $\mathcal{B}:\mathcal{E}_{\M}([-1,1])\longrightarrow \C[[z]]_{\M}$  is surjective if and only if $\widehat{\m}$ satisfies ($\ga_1$).
\end{theo}

We are ready to give the first connection between the growth index $\ga(\M)$ with the surjectivity intervals which holds for arbitrary weight sequences.

\begin{lemma}\label{lemmaSurjectivityImpliessnq}
Let $\M$ be a weight sequence. If $\widetilde{S}_{\M}\neq\emptyset$, then $\M$ has (snq) or, equivalently, $\ga(\M)>0$.
\end{lemma}

\begin{proof}
Let $\widehat{f}=\sum_{p=0}^\infty a_pz^p\in\C[[z]]_{\M}$.
Since there exists $\ga>0$ such that $\widetilde{\mathcal{B}}:\widetilde{\mathcal{A}}_{\M}(S_\ga)\longrightarrow \C[[z]]_{\M}$ is surjective, we may take a function $f_1\in\widetilde{\mathcal{A}}_{\M}(S_\ga)$ such that $\widetilde{\mathcal{B}}(f_1)=\widehat{f}$. A suitable rotation shows that also $\widetilde{\mathcal{B}}:\widetilde{\mathcal{A}}_{\M}(S(\pi,\ga))\longrightarrow \C[[z]]_{\M}$ is  surjective and so there exists a function $f_2\in\widetilde{\mathcal{A}}_{\M}(S(\pi,\ga))$ such that $\widetilde{\mathcal{B}}(f_2)=\widehat{f}$. It is plain to check (by a recursive application of the Mean Value Theorem) that the function
$$
h(x)=f_1(x),\ \ x\in(0,1];\quad h(x)=f_2(x),\ \ x\in[-1,0);\quad
h(0)=a_0,
$$
belongs to $\mathcal{C}^{\infty}([-1,1])$ and $h^{(p)}(0)=p!a_p$ for every $p\in\N$ (see Proposition~\ref{propcotaderidesaasin}). Moreover, considering suitable subsectors of $S_\ga$ (respectively, $S(\pi,\ga)$) containing $(0,1]$ (resp., $[-1,0)$), and again by a double application of Proposition~\ref{propcotaderidesaasin}.(ii), one obtains a constant $A>0$ such that
$$
\sup_{p\in\N_0,\, x\in[-1,1]}\frac{|h^{(p)}(x)|}{A^pp!M_p}<\infty.
$$
Hence, we deduce that the Borel map $\mathcal{B}:\mathcal{E}_{\M}([-1,1]) \longrightarrow \C[[z]]_{\M}$ is also surjective. Since $\M$ is a weight sequence, $\widehat{\M}$ also is, so  by Theorem~\ref{th.Petszche} this surjectivity amounts to the fact that the sequence of quotients of $\widehat{\M}=(p!M_p)_{p\in\N_0}$, namely $\widehat{\m}$, satisfies the condition $(\ga_1)$, which is precisely condition (snq) for $\M$.
\end{proof}

No other result concerning the surjectivity of the Borel map is present in the literature without adding some additional condition on the weight sequence $\M$ in this ultraholomorphic setting.

Our next results, Theorem~\ref{th.SurjectivityNonUniformAsymp} and Theorem~\ref{th.SurjectivityUniformAsymp.plus.dc}, are inspired by statements of J. Schmets and M. Valdivia~\cite[Section~4]{SchmetsValdivia} in the Beurling case.  Although we do not treat this case here, some of their proofs can be adapted to, or suitably modified for, our Roumieu-like spaces.\par

While the aforementioned authors impose condition (dc) on the sequence $\M$, i.e., there exists $A>0$ such that $M_{p+1}\leq A^p M_p$ for every $p\in\N_0$, we will show that, in some cases, one can obtain some information without it.\par

In the course of our arguments we will need to introduce suitable ultradifferentiable classes (the notations again differ from those in~\cite{SchmetsValdivia}):\par

For a natural number $r\in\N$ and a sequence $\M$, we consider the space $\mathcal{N}_{r,\M}([0,\infty))$ of functions $f\in\mathcal{C}^{\infty}([0,\infty))$ such that
\begin{enumerate}[(a)]
 \item $f^{(pr+j)}(0)=0$ for every $p\in\N_0$ and $j\in\{1,\dots,r-1\}$ (this condition is empty when $r=1$),
 \item there exists a constant $A>0$ for which
$$
\sup_{p\in\N_0,\, x\in[0,\infty)}\frac{|f^{(pr)}(x)|}{A^pp!M_p}<\infty.
$$
\end{enumerate}

The subspace of $\mathcal{N}_{r,\M}([0,\infty))$ consisting of those functions with support contained in $[0,1]$ will be denoted by $\mathcal{L}_{r,\M}([0,\infty))$.

Similarly, we introduce the space $\mathcal{E}_{r,\M}([0,1])$ of functions $f\in\mathcal{C}^{\infty}([0,1])$ such that
\begin{enumerate}[(a)]
 \item $f^{(pr+j)}(0)=0$ for every $p\in\N_0$ and $j\in\{1,\dots,r-1\}$ (this condition is empty when $r=1$),
 \item there exists a constant $A>0$ for which
$$
\sup_{p\in\N_0,\, x\in[0,1]}\frac{|f^{(pr)}(x)|}{A^pp!M_p}<\infty.
$$
 \end{enumerate}

Note that these spaces  coincide with the classical ones for $r=1$.
In this context, it is natural to consider the next auxiliary sequence.

\begin{defi}
Given a sequence $\M$ and $r\in\N$, its \textit{$r-$interpolating sequence} $\mathbb{P}_{r,\M}=\mathbb{P}=(P_n)_{n\in\N_0}$ is defined by
$$
P_{kr+j}=\left(M_k^{r-j} M_{k+1}^{j}\right)^{1/r}, \quad k\in\N_0,\ j\in\{0,\dots,r\}.
$$
Note that with $j=r$ for $k$ and $j=0$ for $k+1$  we obtain the same value.
As it was pointed out in~\cite{SchmetsValdivia}, a simple computation leads to
\begin{enumerate}[(i)]
\item$\mathbb{P}_{1,\M}=\M$,
 \item $P_{kr}=M_k$ for every $k\in\N_0$,
 \item $p_{kr+j}=(m_k)^{1/r}$ for all $k\in\N_0$ and $j\in\{0,\dots,r-1\}$,
 \item If $\M$ is a weight sequence, then $\mathbb{P}$ also is.
\end{enumerate}
\end{defi}

We also deduce the following relation for their injectivity indices.

 \begin{lemma}\label{lemma.P.interpolating.omega}
Let $\M$ be a sequence and $r\in\N$. Then
$$\o(\M) = r\o(\mathbb{P}).$$
\end{lemma}
\begin{proof}
Fix $j\in\{0,\dots,r-1\}$ , the lemma is deduced from the next calculation
\begin{align*}
 \o(\M)&=\liminf_{k\to\infty} \frac{\log m_k}{\log k}= r \liminf_{k\to\infty} \frac{\log(m_k)^{1/r}}{\log k}=
r \liminf_{k\to\infty} \frac{\log p_{kr+j}}{\log (kr+j)} \frac{\log (kr+j)}{\log(k)}\\
&= r \liminf_{k\to\infty} \frac{\log p_{kr+j}}{\log (kr+j)}.
\end{align*}
\end{proof}

The introduction of this $r-$interpolating sequence is motivated by the following estimates, independently obtained by A. Gorny and H. Cartan (see ~\cite[Sect.\ 6.4.IV]{Mandelbrojt}).

\begin{lemma}\label{lemma.Gorny.Cartan}
 If $f\in\mathcal{C}^r ([-1,1])$ for some $r\in\N$ and
 $$Q_0:=\sup_{x\in [-1,1]}|f(x)|, \quad\text{and} \quad Q_r:=\sup_{x\in[-1,1]}|f^{(r)}(x)|,$$
then
$$\sup_{x\in[-1,1]}|f^{(j)}(x)|\leq (8er/j)^j \max(Q_0^{1-j/r},Q_r^{j/r},(r/2)^j Q_0).$$
for every $j\in\{1,\dots,r-1\}$.
\end{lemma}

We will employ the integral representation for the reciprocal Gamma function, usually referred
to as \textit{Hankel's formula} (see \cite[p. 228]{balserutx}):
\begin{equation*}
 \frac{1}{\Gamma (z)} =\frac{1}{2\pi i}\int_{\ga_\phi} w^{-z} e^{w} dw
\end{equation*}
for all $z\in \C$ where $\ga_{\phi}$ is a path consisting of a half-line in direction $-\phi\pi/2$ (for any $\phi\in(1,2)$) with end point $w_0$ on the ray $\arg(w)=-\phi\pi/2$ then the circular arc $|w|=|w_0|$ from $w_0$ to the point $w_1$ on the ray $\arg(w)=\phi\pi/2$ (traversed anticlockwise), and finally the half-line starting at $w_1$ in direction $\phi\pi/2$.
Now, for every $\b \in (1, 3/2)$ and any $t\in S_{(\b-1)/2}$, we define
$$\phi_{\b,t}:=\b+2\arg(t)/\pi \in ((\b+1)/2,(3\b-1)/2)\en (1,7/4).$$
Hence, the change of variables $u=t/w$ maps $\ga_{\phi_{\b,t}}$ into  $\delta_{\b}$ which is a path consisting of a segment from the origin to a point $u_0$ with $\arg(u_0)=\b\pi/2$, then the circular arc $|u|=|u_0|$ from $u_0$ to
the point $u_1$ on the ray $\arg(u)=-\b\pi/2$ (traversed clockwise), and
finally the segment from $u_1$ to the origin. Therefore, for every $z\in\C$ and all $t\in S_{(\b-1)/2}$ we have that
\begin{equation}\label{eq.Hankel.formula}
 \frac{t^{z-1}}{\Gamma (z)} = \frac{-1}{2\pi i}\int_{\delta_\b} u^{z-1} e^{t/u} \frac{du}{u}.
\end{equation}

Our first result is obtained as a consequence of the next proposition and the proof is inspired by Theorem 4.6 in~\cite{SchmetsValdivia}.

\begin{prop}[\cite{SchmetsValdivia}, Prop.\ 5.1]\label{prop.51.Lr.SchmetsValdivia}
Let $\M$ be a sequence such that $\widehat{\M}$ is a weight sequence and $r\in\N$. If
the restriction map
$$\mathcal{B}_{r}:\mathcal{L}_{r,\M}([0,\infty))\longrightarrow \C[[z]]_{\M}$$
 sending $f$ to the formal power series $\sum^\infty_{p=0} (f^{(pr)}(0)/p!) z^p$  is surjective, then $\widehat{\m}$ satisfies ($\ga_r$).
\end{prop}

\begin{theo}\label{th.SurjectivityNonUniformAsymp}
Let $\M$ be a weight sequence. 
\begin{itemize}
\item[(i)] Let $\alpha>0$, $\alpha\notin\N$, be such that $\widetilde{\mathcal{B}}:\widetilde{\mathcal{A}}_{\M}(S_{\alpha})\to \C[[z]]_\M$ is surjective. Then, $\ga(\M)>\lfloor \alpha \rfloor $.
\item[(ii)] If we have that $\widetilde{S}_{\M}=(0,\infty)$, then $\ga(\M)=\infty$.
\end{itemize}
\end{theo}

\begin{proof}
(i) Consider first the case $\alpha\in(0,1)$.
Then, it suffices to apply Lemma~\ref{lemmaSurjectivityImpliessnq} to obtain that
$\M$ has (snq), or equivalently 
$\gamma(\M)>0=\lfloor\alpha\rfloor$, as desired.

Suppose now that $\alpha>1$ and put $r=\lfloor\alpha\rfloor$, a positive natural number. Firstly, for $\widecheck{\M}=(M_p/p!)_{p\in\N_0}$ we will prove that the restriction map $\mathcal{B}_{r}:\mathcal{E}_{r,\widecheck{\M}}([0,1])\longrightarrow \C[[z]]_{\widecheck{\M}}$  is surjective.
Since $r\notin\N$, we may choose two numbers $\beta_1,\beta_2$ with
$$
1<\beta_1<\beta_2<\min\{\frac{\alpha}{r},\frac{3}{2}\}.
$$
Given $\widehat{g}= \sum^\infty_{p=0} a_p z^p \in\C[[z]]_{\widecheck{\M}}$, we write $b_p:=a_p p!$ for all $p\in\N_0$, and there exist $C_0,A_0>0$ such that
\begin{equation*}
|b_p|\le C_0A_0^pp!\widecheck{M}_p=C_0A_0^pM_p,\quad p\in\N_0.
\end{equation*}
Hence, the formal Laplace transform of $\widehat{g}$, defined by  $\widehat{f}:=\widehat{L}\widehat{g}=\sum_{p=0}^{\infty}b_pz^p$ belongs to $\C[[z]]_\M$. By hypothesis, there exists $\psi\in\widetilde{\mathcal{A}}_{\M}(S_{\alpha})$ such that $\widetilde{\mathcal{B}}(\psi)=\widehat{f}$. Hence, given $\beta_2$ and $R>1$, there exist $C,A>0$ such that for every $p\in\N_0$ one has
\begin{equation}\label{equaAsympExpanPsiProofDCbis}
\Big|\psi(z)-\sum_{k=0}^{p-1}b_kz^k \Big|\le CA^pM_{p}|z|^p,\qquad z\in S(0,r\beta_2,R^r).
\end{equation}

The function $\varphi:S_{\alpha/r}\to\C$ given by $\varphi(u)=\psi(u^{r})$, is well defined and holomorphic in $S_{\alpha/r}$, which contains $S_{\beta_2}$ as a proper unbounded subsector. Moreover, according to~\eqref{equaAsympExpanPsiProofDCbis} for $p=0$, for every $w\in S(0,{\beta_2},R)$ one has
\begin{equation}\label{equaBoundsVarphiProofDCbis}
\left|\varphi(u)\right|=|\psi(u^{r})|\le CM_{0}.
\end{equation}
We consider now a path
$\delta_{\beta_1}$ in $S(0,{\beta_2},R)$ like the ones used in the classical Borel transform, made up of a segment $\delta_1$ from the origin to a point $u_0$ with $|u_0|=R_0<R$ and $\arg(u_0)=\pi\beta_1/2$, then the circular arc $\delta_2$, traversed clockwise on the circumference $|u|=R_0$ and going from $u_0$ to
the point $u_1$ on the ray $\arg(u_1)=-\pi\beta_1/2$, and
finally the segment $\delta_3$ from $u_1$ to the origin.

Define the function $f:S_{(\beta_1-1)/2}\to\C$ given by
$$
f(t)=\frac{-1}{2\pi i}\int_{\delta_{\beta_1}}e^{t/u} \varphi(u) \frac{du}{u}.
$$
Observe that $\varphi(u)$ is holomorphic and bounded at 0 in $S(0,{\beta_2},R)$, and for every $t\in S_{(\beta_1-1)/2}$ one may easily check that $t/u$ runs over a half-line in the open left half-plane and tends to infinity as $u$ runs over any of the segments $\delta_1$ or $\delta_3$ and tends to 0. Hence, $f$ is holomorphic in the sector $S_{(\beta_1-1)/2}$. We note that, by virtue of Cauchy's theorem, the value assigned to $R_0$ in the definition of $\delta_{\beta_1}$ is irrelevant for the value of $f$.

Let us fix in the following estimations  some $t\in S(0,(\beta_1-1)/2,R)$ and some natural number $p\in\N$. Hankel's formula~\eqref{eq.Hankel.formula} for $z=kr+1$ allows us to write
\begin{align}\label{equaRemainderProofDCbis}
f(t)-\sum_{k=0}^{p-1}b_k\frac{t^{kr}}{(kr)!}&=
-\frac{1}{2\pi i}\int_{\delta_{\beta_1}}e^{t/u}
\left(\varphi(u)-\sum_{k=0}^{p-1}b_ku^{kr}\right)\,\frac{du}{u} \nonumber \\
&=-\frac{1}{2\pi i}\sum_{j=1}^3\int_{\delta_j}e^{t/u}
\left(\varphi(u)-\sum_{k=0}^{p-1}b_ku^{kr}\right)\,\frac{du}{u}.
\end{align}
Taking into account~\eqref{equaAsympExpanPsiProofDCbis}, for every $u\in S(0,\beta_2,R)$ we have
\begin{equation}\label{equaRemainderPhiProofDCbis}
\left|\varphi(u)- \sum_{k=0}^{p-1}b_ku^{kr}\right|=
\left|\psi(u^{r})- \sum_{k=0}^{p-1}b_k(u^{r})^k\right|\le
CA^pM_p|u|^{pr}.
\end{equation}
So,
if we choose $R_0=|t|/p<R$, we may apply \eqref{equaRemainderPhiProofDCbis} and see that
\begin{equation}\label{equacotadelta2}
\left|\int_{\delta_2}e^{t/u}
\left( \varphi(u)-\sum_{k=0}^{p-1}b_ku^{kr}\right)\, \frac{du}{u}\right|
\le
\pi\beta_1 e^pCA^pM_p\left(\frac{|t|}{p}\right)^{pr}.
\end{equation}
On the other hand, by the same estimates \eqref{equaRemainderPhiProofDCbis} and by the choice made for $R_0$, for $j=1,3$ we have
\begin{align}\label{equacotadelta13}
\left|\int_{\delta_j}e^{t/u}
\left( \varphi(u)-\sum_{k=0}^{p-1}b_ku^{kr}\right)\,\frac{du}{u}\right|
&\le
CA^pM_p\int_0^{|t|/p}s^{pr}|e^{t/(se^{\pm i\pi\beta_1/2})}|\,\frac{ds}{s}
\nonumber\\
&\le CC_1A^pM_p\left(\frac{|t|}{p}\right)^{pr},
\end{align}
where $C_1$ is a constant, independent of both $t$ and $p$, given by
\begin{align*}
C_1&=\sup_{t\in S(0,(\beta_1-1)/2,R),\ p\in \N} \int_0^{|t|/p}|e^{t/(se^{\pm i\pi\beta_1/2})}|\,\frac{ds}{s}\\
&=\sup_{t\in S(0,(\beta_1-1)/2,R),\ p\in \N} \int_0^{|t|/p}e^{|t|\cos(\arg(t)\mp\pi\beta_1/2)/s}\,\frac{ds}{s}\\
&\le\sup_{|t|<R,\ p\in\N} \int_0^{|t|/p}e^{-|t|\cos(\pi(\beta_1-1)/4)/s}\,\frac{ds}{s}
=\sup_{p\in\N} \int_0^{1/p}e^{-\cos(\pi(\beta_1-1)/4)/u}\,\frac{du}{u}\\
&\le \int_0^{1}e^{-\cos(\pi(\beta_1-1)/4)/u}\,\frac{du}{u}<\infty.
\end{align*}
According to
\eqref{equaRemainderProofDCbis}, \eqref{equacotadelta2} and \eqref{equacotadelta13}, and using Stirling's formula, we find that there exist constants $C_2,A_2>0$ such that for every $p\in\N$ and $t\in S(0,(\beta_1-1)/2,R)$ one has
\begin{equation}\label{equaAsympExpanfProofDCbis}
\left|f(t)-\sum_{k=0}^{p-1}b_k\frac{t^{kr}}{(kr)!}\right|\le
 C_2A_2^p\frac{M_p}{(pr)!}|t|^{pr}.
\end{equation}
This last estimation also holds for $p=0$, in a similar way, taking $R_0=|t|$ and using the definition of $f$ and \eqref{equaBoundsVarphiProofDCbis}.
Hence one can show that $f$ admits the series $\sum_{p=0}^{\infty}b_pt^{pr}/(pr)!$ as its asymptotic expansion as $t$ tends to 0 in the sector (if $r\geq 2$ observe that for $(p-1)r+1\le n<pr$ we have $|t|^{pr}\le |t|^n$ whenever $|t|\le 1$). It is then a standard fact that for every $m\in\N_0$ and every proper subsector $T$ of $S(0,(\beta_1-1)/2,R)$ there exists
\begin{equation}\label{equaLimitDerivat0ProofDCbis}
\lim_{t\to 0,\ t\in T}f^{(m)}(t)=\begin{cases}
b_p&\text{if $m=pr$ for some natural number $p\in\N_0$,}\\
0&\text{otherwise.}
\end{cases}
\end{equation}
Finally, we define the function $F\colon[0,1]\to\C$ given by $F(t)=f(t)$ for $t\in(0,1]$, $F(0)=b_0$. Since $f$ is holomorphic in $S(0, (\b_1-1)/2, R)$ and we have~\eqref{equaLimitDerivat0ProofDCbis}, we immediately deduce that $F$ belongs to $\mathcal{C}^\infty([0,1])$ and
\begin{equation*}
F^{(m)}(0)=\begin{cases}
b_p&\text{if $m=pr$ for some $p\in\N_0$,}\\
0&\text{otherwise.}
\end{cases}
\end{equation*}
Moreover, we may take $\varepsilon>0$ such that for every $t\in(0,1]$ the disk $D(t,\varepsilon t)$ is contained in $S(0,(\beta_1-1)/2,R)$. Then, Cauchy's integral formula together with~\eqref{equaAsympExpanfProofDCbis} allow us to deduce that for every $p\in\N_0$,
\begin{align*}
|F^{(pr)}(t)|&=\left|\left(f(t)- \sum_{k=1}^{p-1}b_k\frac{t^{kr}}{(kr)!}\right)^{(pr)}\right|\le (pr)!\left(\frac{1+\varepsilon}{\varepsilon}\right)^{pr} \frac{C_2A_2^pM_p}{(pr)!}=C_3A_3^pM_p.
\end{align*}
In conclusion, $F\in\mathcal{E}_{r,\widecheck{\M}}([0,1])$ and $\mathcal{B}_r(F)=\widehat{g}$. So, $S$ is surjective.\par

Secondly, according to Theorem~\ref{theoNotBijectivity} the map $\widetilde{\mathcal{B}}:\widetilde{\mathcal{A}}_{\M}(S_{\alpha})\to \C[[z]]_\M$ is not injective, this means by Theorem~\ref{coroGenerWatsonLemmaSectorialRegions} that $\a\leq \o(\M)$, then
$r=\lfloor \a \rfloor<\o(\M)$ because $\a\notin\N$. By Lemma~\ref{lemma.P.interpolating.omega} and~\eqref{equaPropertiesIndexOmega},
if $\mathbb{P}=\mathbb{P}_{r,\M}$ we have that
$$\o(\widecheck{\mathbb{P}})=\o(\mathbb{P}_{r,\M})-1=\o(\M)/r-1>0.$$
Hence, since $\mathbb{P}$ is (lc), one may take into account~\eqref{equaOmegaExponConvergenceWidehatM} and deduce that $\widecheck{\mathbb{P}}$ has (nq), so by the Denjoy-Carleman theorem (see~\cite[Ch.\ 1]{Hormander1990}) there exists a $\mathcal{C}^{\infty}$ nonnegative function $\varphi$ in $\R$ with support contained in $[-1,1]$ and which takes the value 1 in a neighborhood of 0, such that there exists $A>0$ with
$$
\sup_{t\in\R,\ n\in\N_0}\frac{|\varphi^{(n)}(t)|}{A^nP_n}<\infty.
$$
Applying the Gorny-Cartan estimates of Lemma~\ref{lemma.Gorny.Cartan}, for every $h\in\mathcal{E}_{r,\widecheck{\M}}([0,1])$ one can check that the product $\varphi h$ belongs to $\mathcal{L}_{r,\widecheck{\M}}([0,\infty))$ and, moreover, $(\varphi h)^{(p)}(0)=h^{(p)}(0)$ for every $p\in\N_0$.\par

Since $\mathcal{B}_{r}:\mathcal{E}_{r,\widecheck{\M}}([0,1])\to \C[[z]]_{\widecheck{\M}}$ is surjective, we deduce that $\mathcal{B}_{r}:\mathcal{L}_{r,\widecheck{\M}}([0,\infty))\longrightarrow \C[[z]]_{\widecheck{\M}}$ also is. By Proposition~\ref{prop.51.Lr.SchmetsValdivia}, we conclude  that $\bm$ satisfies $(\gamma_r)$, what amounts to $\gamma(\M)>r=\lfloor\alpha\rfloor$.

(ii) It is an immediate consequence of~(i).
\end{proof}

\begin{coro}\label{coro.SurjectIntervalNonUniformAsympt}
Whenever $\M$ is a weight sequence, if $\ga(\M)<\infty$ one always has
$$\widetilde{S}_{\M} \subseteq (0,\lfloor \ga(\M) \rfloor +1].
$$
In case $\gamma(\M)\in\N$, then $\widetilde{S}_{\M} \subseteq (0,\ga(\M)+1)$. Note that if $\ga(\M)=\infty$, the previous theorem does not provide any relevant information.
\end{coro}

\begin{proof}
The case $\widetilde{S}_{\M}=\emptyset$ is trivial. So, we treat the case
in which  the surjectivity interval is not empty, what according to Lemma~\ref{lemmaSurjectivityImpliessnq} implies  $\gamma(\M)>0$.

Let $\alpha\in \widetilde{S}_{\M}$. On the one hand, if $\alpha\notin\N$, by Theorem~\ref{th.SurjectivityNonUniformAsymp} we have $\lfloor\alpha  \rfloor<\ga(\M)$, and so $\alpha-1<\lfloor\alpha  \rfloor\le\lfloor\ga(\M)\rfloor$, from where $\alpha<\lfloor\ga(\M)  \rfloor+1$. On the other hand, if $\alpha\in\N$ then we can apply Theorem~\ref{th.SurjectivityNonUniformAsymp} for any $\beta\in(\alpha-1,\alpha)$ (since $\beta\in \widetilde{S}_{\M}$ too) and deduce that $\alpha-1=\lfloor\beta  \rfloor<\ga(\M)$, hence $\alpha<\ga(\M)+1$. We deduce that  $\alpha\le \lfloor\ga(\M)+1  \rfloor=\lfloor\ga(\M)  \rfloor+1$, except in case $\gamma(\M)\in\N$, where moreover $\alpha$ cannot coincide with $\ga(\M)+1$. The conclusion easily follows.
\end{proof}

\begin{rema}\label{rema.intervals.Wseq}
Summing up,  for a weight sequence $\M$ and taking into account~\eqref{equaContentionSurjectIntervals} and Theorem~\ref{theoNotBijectivity} we see that:
 \begin{enumerate}[(i)]
  \item if $\ga(\M)=0$ (equivalently, if $\M$ has not (snq)) then $S_{\M}=\widetilde{S}^u_{\M}=\widetilde{S}_{\M}=\emptyset$.
  \item if $\ga(\M)\in(0,\infty)$ and
   \begin{enumerate}
    \item  $\ga(\M)\notin\N$, then  $S_{\M}\en\widetilde{S}^u_{\M}\en\widetilde{S}_{\M}\en(0,\lfloor \ga(\M) \rfloor+1]\cap(0,\o(\M)]$,
    \item $\ga(\M)\in\N$, then $S_{\M}\en\widetilde{S}^u_{\M}\en\widetilde{S}_{\M}\en(0,\ga(\M) +1)\cap(0,\o(\M)]$.
   \end{enumerate}
 If $\o(\M)=\infty$, the second interval in these intersections should be taken as $(0,\infty)$.
 \end{enumerate}
\end{rema}

\subsection{Weight sequences satisfying derivation closedness condition}

As it has been pointed out in Remark~\ref{rema.intervals.Wseq}, Corollary~\ref{coro.SurjectIntervalNonUniformAsympt} provides also information about $\widetilde{S}^u_{\M}$.  In order to slightly improve it, one needs to impose (dc), which is a natural condition on the sequence $\M$, in the sense that it guarantees that the ultraholomorphic classes under consideration, consisting of holomorphic functions, are closed with respect to taking derivatives (see Remarks~\ref{rema.alg.prop.ultra.class} and~\ref{rema.image.Borel.map}). We will also need the next result.

\begin{prop}[\cite{SchmetsValdivia}, Prop.\ 5.2]\label{prop.52.Nr.SchmetsValdivia}
Let $r\in\N$ and $\M$ be a sequence such that $\widehat{\M}=(p!M_p)_{p\in\N_0}$ is a weight sequence. If the map $\mathcal{B}_r:\mathcal{N}_{r,\M}([0,\infty))\longrightarrow \C[[z]]_{\M}$
 sending $f$ to the formal power series $\sum^\infty_{p=0} (f^{(pr)}(0)/p!) z^p$ is surjective, then the sequence $\widehat{\bm}=((p+1)m_p)_{p\in\N_0}$ satisfies the condition $(\gamma_r)$.
\end{prop}

Following the ideas in the proof of Proposition 4.6 in~\cite{SchmetsValdivia}, we will be able to deal also with the case $\alpha\in\N$ whenever $\widetilde{\mathcal{B}}:\widetilde{\mathcal{A}}^u_{\M}(S_{\alpha})\to \C[[z]]_\M$ is surjective.

\begin{theo}\label{th.SurjectivityUniformAsymp.plus.dc}
Let $\M$ be a weight sequence satisfying (dc).
\begin{itemize}
\item[(i)] Let $\alpha>0$ be such that $\widetilde{\mathcal{B}}:\widetilde{\mathcal{A}}^u_{\M}(S_{\alpha})\to \C[[z]]_\M$ is surjective. Then, $\ga(\M)>\lfloor \alpha \rfloor $.
\item[(ii)] If we have that $\widetilde{S}^u_{\M}=(0,\infty)$, then $S_{\M} =\widetilde{S}^u_{\M} =\widetilde{S}_{\M}=(0,\infty)$ and $\ga(\M)=\infty$.
\end{itemize}
\end{theo}

\begin{proof}
(i) Consider first the case $\alpha\in(0,1)$, then $\a\in \widetilde{S}^u_{\M} \en\widetilde{S}_{\M}$ and $\a\notin\N$, so
by Theorem~\ref{th.SurjectivityNonUniformAsymp} we conclude that $\ga(\M)>0$.
Note that in this case no use has been made of (dc).

Suppose now that $\alpha\ge 1$ and put $r=\lfloor\alpha\rfloor$, a positive natural number (note that, by Theorem~\ref{th.SurjectivityNonUniformAsymp}, we only would need to consider the case $\a=r\in\N$ but  the proof works anyway).
Our aim is to show that $\mathcal{B}_r:\mathcal{N}_{r,\widecheck{\M}}([0,\infty))\longrightarrow \C[[z]]_{\widecheck{\M}}$ is surjective.

Given $\widehat{g}=\sum^\infty_{p=0} a_p z^p\in\C[[z]]_{\widecheck{\M}}$, we write $b_p:=a_p p!$ for all $p\in\N_0$ and we see that there exist $C_0,A_0>0$ such that
\begin{equation}\label{equaBoundsbpProofDC}
|b_p|\le C_0A_0^pp!\widecheck{M_p}=C_0A_0^pM_p,\quad p\in\N_0.
\end{equation}
Consider the formal power series $\widehat{f}=\sum_{p=0}^{\infty}(-1)^{pr}b_pz^p\in\C[[z]]_\M$. By hypothesis, there exists $\psi\in\widetilde{\mathcal{A}}^u_{\M}(S_{\alpha})$ such that $\widetilde{\mathcal{B}}(\psi)=\widehat{f}$, and so there exist $C,A>0$ such that for every $p\in\N_0$ one has
\begin{equation}\label{equaAsympExpanPsiProofDC}
\Big|\psi(z)-\sum_{k=0}^{p-1}(-1)^{kr}b_kz^k \Big|\le CA^pM_{p}|z|^p,\qquad z\in S_\alpha.
\end{equation}
The function $\varphi:S_{\alpha/r}\to\C$ given by $\varphi(w)=\psi(w^{-r})-b_0$, is well defined and holomorphic in $S_{\alpha/r}\supseteq S_1$. Moreover, according to~\eqref{equaAsympExpanPsiProofDC} for $p=1$, for every $w\in S_1$ one has
\begin{equation}\label{equaBoundsVarphiProofDC}
\left|\frac{\varphi(w)}{w}\right|=\frac{1}{|w|}|\psi(w^{-r})-b_0|\le \frac{CAM_{1}}{|w|^{r+1}}.
\end{equation}
So, the function $f:\R\to\C$ given by
\begin{equation*}
f(t)=\frac{1}{2\pi i}\int_{1-\infty\,i}^{1+\infty\,i}e^{tu}
\frac{\varphi(u)}{u}\,du
\end{equation*}
is well defined and continuous on $\R$. By the classical Hankel formula~\eqref{eq.Hankel.formula} for the reciprocal Gamma function, for every natural number $p\ge 2$ and every $t\in\R$ we may write
\begin{equation}\label{equaRemainderProofDC}
f(t)-\sum_{k=1}^{p-1}(-1)^{kr}b_k\frac{t^{kr}}{(kr)!}=
\frac{1}{2\pi i}\int_{1-\infty\,i}^{1+\infty\,i}e^{tu}
\left(\frac{\varphi(u)}{u}- \sum_{k=1}^{p-1}\frac{(-1)^{kr}b_k}{u^{kr+1}}\right)\,du.
\end{equation}
Since, again by~\eqref{equaAsympExpanPsiProofDC}, we have
\begin{equation}\label{equaRemainderPhiProofDC}
\left|\frac{\varphi(u)}{u}- \sum_{k=1}^{p-1}(-1)^{kr}b_k\frac{1}{u^{kr+1}}\right|=
\frac{1}{|u|}\left|\psi(u^{-r})- \sum_{k=0}^{p-1}(-1)^{kr}b_k(u^{-r})^k\right|\le
\frac{CA^pM_p}{|u|^{pr+1}}
\end{equation}
for every $u\in S_1$, we can apply Leibniz's theorem for parametric integrals and deduce that the function
$$
f(t)-\sum_{k=1}^{p-1}(-1)^{kr}b_k\frac{t^{kr}}{(kr)!}
$$
belongs to $\mathcal{C}^{pr-1}(\R)$. Moreover, all of its derivatives of order $m\le pr-1$ at $t=0$ vanish. This fact can be checked by differentiating the right-hand side of~\eqref{equaRemainderProofDC} $m$ times under the integral sign, evaluating at $t=0$, and then computing the integral by means of Cauchy's theorem. For that, consider the paths $\Gamma_s$, $s>0$, consisting of the arc of circumference centered at 1, joining $1+si$ and $1-si$ and passing through $1+s$, and the segment $[1-si,1+si]$. It is plain to check that $\int_{\Gamma_s} u^{m-1} (\varphi(u)-\sum_{k=1}^{p-1} (-1)^{kr}b_k u^{-kr} )du=0$, and applying~\eqref{equaRemainderPhiProofDC} a limiting process when $s\to\infty$ leads to the conclusion.

As $p$ is arbitrary, we have that $f\in\mathcal{C}^{\infty}(\R)$ and, moreover,
$$
f^{(m)}(0)=\begin{cases}
  (-1)^{pr}b_p&\textrm{if $m=pr$ for some $p\ge 1$},\\
  0&\textrm{otherwise.}
\end{cases}
$$
Finally, we define the function
$$
F(t)=b_0+f(-t),\quad t\ge 0.
$$
Obviously, $F\in\mathcal{C}^{\infty}([0,\infty))$ and
$F^{(pr)}(0)=b_p$, $p\in\N_0$; $F^{(m)}(0)=0$ otherwise.
In order to conclude, we estimate the derivatives of $F$ of order $pr$ for some $p\in\N_0$. For $p=0$ and $t\ge 0$, we take into account~\eqref{equaBoundsbpProofDC} and \eqref{equaBoundsVarphiProofDC} in order to obtain that
\begin{equation}\label{equaBoundFtProofDC}
|F^{(0)}(t)|\le |b_0|+\frac{1}{2\pi}\int_{-\infty}^{\infty}e^{-t}
\frac{CAM_{1}}{|1+yi|^{r+1}}\,dy\le
C_0+\frac{CAM_1}{2\pi}
\int_{-\infty}^{\infty}\frac{1}{(1+y^2)^{(r+1)/2}}\,dy,
\end{equation}
and so $F$ is bounded.
For $p\ge 1$ we may write formula~\eqref{equaRemainderProofDC}  evaluated at $-t$ as
\begin{equation*}
f(-t)-\sum_{k=1}^{p}b_k\frac{t^{kr}}{(kr)!}=
\frac{1}{2\pi i}\int_{1-\infty\,i}^{1+\infty\,i}e^{-tz}
\left(\frac{\varphi(z)}{z}- \sum_{k=1}^{p}\frac{(-1)^{kr}b_k}{z^{kr+1}}\right)\,dz.
\end{equation*}
Then,
\begin{align*}
F^{(pr)}(t)&=b_p+\left(f(-t)- \sum_{k=1}^{p}b_k\frac{t^{kr}}{(kr)!}\right)^{(pr)}(t)\\
&=
b_p+\frac{1}{2\pi i}\int_{1-\infty\,i}^{1+\infty\,i}e^{-tz}
(-z)^{pr}\left(\frac{\varphi(z)}{z}- \sum_{k=1}^{p}\frac{(-1)^{kr}b_k}{z^{kr+1}}\right)\,dz,
\end{align*}
and we may apply~\eqref{equaBoundsbpProofDC}, and~\eqref{equaRemainderPhiProofDC} in order to obtain
\begin{equation}\label{equaBoundDerivFtProofDC}
|F^{(pr)}(t)|\le C_0A_0^pM_p+\frac{CA^{p+1}M_{p+1}}{2\pi} \int_{-\infty}^{\infty}\frac{1}{(1+y^2)^{(r+1)/2}}\,dy.
\end{equation}
From \eqref{equaBoundFtProofDC} and~\eqref{equaBoundDerivFtProofDC}, and since $\M$ satisfies (dc), we deduce that there exist $C_1,A_1>0$ such that for every $p\in\N_0$ one has
$$
|F^{(pr)}(t)|\le C_1A_1^pM_p=C_1A_1^p p!\widecheck{M_p},\quad t\ge 0,
$$
and so $F\in\mathcal{N}_{r,\widecheck{\M}}([0,\infty))$ and $\mathcal{B}_r(F)=\widehat{g}$. In conclusion, $\mathcal{B}_r$ is surjective as desired, and by Proposition~\ref{prop.52.Nr.SchmetsValdivia} we deduce that $\bm$ satisfies $(\gamma_r)$, what amounts to $\gamma(\M)>r=\lfloor\alpha\rfloor$.

(ii) The fact that all the intervals of surjectivity are $(0,\infty)$ is an easy consequence of~\eqref{equaContentionSurjectIntervals} and Proposition~\ref{propcotaderidesaasin}.(iii), while $\gamma(\M)=\infty$ stems from (i).
\end{proof}

\begin{coro}\label{coroSurjInterUnifAsymptplusdc}
Whenever $\M$ is a weight sequence satisfying (dc), one has
$$S_\M\subseteq \widetilde{S}^u_{\M} \subseteq (0,\lfloor \ga(\M) \rfloor +1).
$$
If moreover $\ga(\M)\in\N$, then $S_\M\subseteq \widetilde{S}^u_{\M} \subseteq (0,\ga(\M))$.
\end{coro}

\begin{proof}
The arguments are similar to those in the proof of Corollary~\ref{coro.SurjectIntervalNonUniformAsympt}.
The case $\widetilde{S}^u_{\M}=\emptyset$ is trivial. Otherwise, $\widetilde{S}_{\M}\neq\emptyset$ and, by Lemma~\ref{lemmaSurjectivityImpliessnq}, $\gamma(\M)>0$.

Let $\alpha\in \widetilde{S}^u_{\M}$. By Theorem~\ref{th.SurjectivityUniformAsymp.plus.dc} we have $\lfloor\alpha  \rfloor<\ga(\M)$, and so $\alpha<\lfloor\alpha  \rfloor +1\le\lfloor\ga(\M)\rfloor +1$, which is the first statement. In case $\gamma(\M)\in\N$, the condition $\lfloor\ga(\M)  \rfloor<\ga(\M)$ does not hold, and so $\ga(\M)\notin \widetilde{S}^u_{\M}$ and the interval $\widetilde{S}^u_{\M}$ has to be contained in $(0,\ga(\M))$.
\end{proof}

Recall that if $\M$ has not (snq) the problem is solved (see Remark~\ref{rema.intervals.Wseq}).
Let $\M$ be (lc), (snq) and (dc) (the first two conditions imply that $\M$ is a weight sequence). Then $\ga(\M)\in(0,\infty]$, and we have the situation described in Table~\ref{tableSurjectivity.lc.snq.dc}, with the corresponding conventions if $\ga(\M)=\infty$ or $\o(\M)=\infty$. With the same assumptions, one might be able to show at least  that $\widetilde{S}_{\M} \en\widetilde{S}^u_{\M}\en (0,\ga(\M))$ and $\widetilde{S}_{\M}\en(0,\ga(\M)]$ but it seems that a technique that only employs the properties of  the spaces $\mathcal{E}_{r,\M}$, $\mathcal{N}_{r,\M}$ and $\mathcal{L}_{r,\M}$  is not sufficient.\par

We mention that there exist sequences that are not strongly regular such that $\ga(\M),\o(\M)\in(0,\infty)$, and these values still are referring to some concrete openings in the injectivity and surjectivity problems.

\begin{table}[h]
\centering
\begin{tabular}{l|l|}
\cline{1-2}
                   \multicolumn{1}{|c}{$\ga(\M)\in\mathbb{N}$} \rule{0pt}{1.1\normalbaselineskip}     &
                   \multicolumn{1}{|c|}{ $\ga(\M)\in\mathbb{R}\backslash \N$} \\ \hline
     \multicolumn{1}{|c}{ $ S_{\M}\subseteq(0,\ga(\M))$} \rule{0pt}{1.1\normalbaselineskip} &     \multicolumn{1}{|c|}{ $S_{\M}\subseteq(0,\lfloor \ga(\M) \rfloor+1)\cap (0,\o(\M)]$}       \\ \hline
    \multicolumn{1}{|c}{ $ \widetilde{S}^u_{\M}\subseteq(0,\ga(\M))$} \rule{0pt}{1.1\normalbaselineskip} &   \multicolumn{1}{|c|}{ $ \widetilde{S}^u_{\M}\subseteq(0,\lfloor \ga(\M) \rfloor+1)\cap (0,\o(\M)]$}                 \\ \hline
\multicolumn{1}{|c}{ $ \widetilde{S}_{\M} \subseteq (0, \ga(\M) +1)\cap(0,\o(\M)]$}  \rule{0pt}{1.1\normalbaselineskip} &  \multicolumn{1}{|c|}{ $ \widetilde{S}_{\M} \subseteq (0,\lfloor \ga(\M) \rfloor+1]\cap(0,\o(\M)]$} \\ \hline
\end{tabular}
\caption{Surjectivity intervals when $\M$ is (lc), (snq) and (dc).}
\label{tableSurjectivity.lc.snq.dc}
\end{table}

\subsection{Strongly regular sequences}\label{subsectSurjStronglyRegSeq}

We need to impose more conditions on the sequence $\M$ in order to get extra information about surjectivity. We recall that $\M$ is said to be strongly regular if is (lc), (snq) and (mg). As commented before, the first two conditions are natural in this context, and moderate growth, which is stronger than (dc), is our additional assumption. We recall that
a (lc) sequence has (mg) if, and only if, $\sup_{p\in\N} m_p/M^{1/p}_p <\infty$ (see~\cite[Lemma~5.3]{PetzscheVogt}).
Hence, since for a (lc) and (mg) sequence one has, with Landau's notation, $\log(M_p)=O(p\log(p))$ as $p$ tends to infinity (see~\cite[Theorem~2]{Matsumoto}), using (i) we deduce that
$$
\o(\M)=\liminf_{p\to\infty}\frac{\log(m_{p})}{\log(p)}\le \liminf_{p\to\infty}\frac{\log(M_{p})}{p\log(p)}<\infty.
$$
With this, Proposition~\ref{propComparisonIndices} and the equivalence of (snq) and the condition $\ga(\M)>0$, for a strongly regular sequence one always has $0<\ga(\M)\le \o(\M)<\infty$ (see also~\cite{PhDJimenez,JimenezSanzSchindlIndex}).

The main known result regarding surjectivity for strongly regular sequences was provided by V. Thilliez~\cite[Theorem\ 3.2.1]{thilliez}.

\begin{theo}[\cite{thilliez}, Theorem~3.2.1] \label{th.ThilliezSurjectivity}
Let $\M$ be a strongly regular sequence and $0<\ga<\ga(\M)$. Then
 there exists $d\ge1$ such that for every $A>0$ there is a linear continuous operator
$$
T_{\M,A,\gamma}:\C[[z]]_{\M,A} \to\mathcal{A}_{\M,dA}(S_{\gamma})
$$
such that $\widetilde{\mathcal{B}}\circ T_{\M,A,\gamma}={\emph{Id}}_{\C[[z]]_{\M,A}}$, the identity map in $\C[[z]]_{\M,A}$. Hence, $\widetilde{\mathcal{B}}:\mathcal{A}_{\M}(S_\gamma)\longrightarrow \C[[z]]_{\M}$ is surjective.
\end{theo}

Except in the classical Gevrey classes, no information about the optimality of $\ga(\M)$ was provided. Our next attempt will be to obtain as much information as possible in this direction. The following result rests on Theorem~\ref{th.SurjectivityUniformAsymp.plus.dc} and a ramification argument, what makes us consider only rational values for the constant $r$ below.

\begin{theo}\label{th.SurjStronglyRegGammaMRational}
Let $\M$ be a strongly regular sequence, and let $r\in\mathbb{Q}$, $r>0$ be given. The following assertions are equivalent:
\begin{itemize}
\item[(i)] $r<\ga(\M)$,
\item[(ii)] there exists $d\ge1$ such that for every $A>0$ there is a linear continuous operator
$$
T_{\M,A,r}:\C[[z]]_{\M,A} \to\mathcal{A}_{\M,dA}(S_{r})
$$
such that $\widetilde{\mathcal{B}}\circ T_{\M,A,\gamma}={\emph{Id}}_{\C[[z]]_{\M,A}}$ the identity map in $\C[[z]]_{\M,A}$,
\item[(iii)] the Borel map $\widetilde{\mathcal{B}}:\mathcal{A}_{{\M}}(S_{r})\to \C[[z]]_\M$ is surjective,
\item[(iv)] the Borel map $\widetilde{\mathcal{B}}:\widetilde{\mathcal{A}}^u_{\M}(S_{r})\to \C[[z]]_\M$ is surjective.
\end{itemize}
\end{theo}
\begin{proof}
(i)$\implies$(ii)$\implies$(iii) This is Theorem~\ref{th.ThilliezSurjectivity}.

(iii)$\implies$(iv)  Trivial by contention.

(iv)$\implies$(i) In case $r\in\N$, we use Theorem~\ref{th.SurjectivityUniformAsymp.plus.dc}.(i) and we conclude.

Otherwise, we write $r=p/q$ with $p,q\in\N$ relatively prime, $q\ge 2$. Consider the sequence $\M^{q}=(M_n^{q})_{n\in\N_0}$, which also turns out to be strongly regular (see~\cite[Lemma~1.3.4]{thilliez}). We will prove that $\widetilde{\mathcal{B}}:\widetilde{\mathcal{A}}^u_{{\M}^{q}}(S_{p})\to \C[[z]]_{\M^{q}}$ is surjective, so, again by  Theorem~\ref{th.SurjectivityUniformAsymp.plus.dc}.(i), we see that $p<\ga(\M^{q})$. Hence, we get that $r=p/q<\ga(\M)$, as desired.\par

Let us prove the aforementioned surjectivity. Given $\widehat{f}=\sum_{j=0}^\infty a_jz^j\in\C[[z]]_{\M^{q}}$, there exist $C,A>0$ such that $|a_j|\le CA^jM_j^q$ for every $j\in\N_0$. Let us define a new formal power series $\widehat{g}=\sum_{j=0}^\infty b_jz^j$ with coefficients
$$
b_{qj}=a_j,\ j\in\N_0;\quad b_m=0\text{ otherwise.}
$$
The log-convexity of $\M$ implies that $M_j^q\le M_{qj}$ for every $j$, so we have that
$$
|b_{qj}|\le CA^jM_j^q\le C(A^{1/q})^{qj}M_{qj},
$$
and consequently, $\widehat{g}\in\C[[z]]_{\M}$. By hypothesis, there exists a function $g\in \widetilde{\mathcal{A}}^u_{\M}(S_{r})$ such that $\widetilde{\mathcal{B}}(g)=\widehat{g}$, and so there exist $C_1,A_1>0$ such that for every $z\in S_r$ and $n\in\N_0$ one has
\begin{equation}\label{equaAsympExpangProofSurjGammaRational}
\left|g(z)-\sum_{j=0}^{n-1}b_jz^j\right|\le C_1A_1^nM_n|z|^n.
\end{equation}
Consequently, the function $f\colon S_p\to\C$ given by $f(w)=g(w^{1/q})$ is well-defined and holomorphic in $S_p$. Moreover, for every $w\in S_p$ and $n\in\N_0$ one deduces from~\eqref{equaAsympExpangProofSurjGammaRational} that
\begin{align}\label{equaAsympExpanfProofSurjGammaRational}
\left|f(w)-\sum_{j=0}^{n-1}a_jw^j\right|&=
\left|g(w^{1/q})-\sum_{j=0}^{n-1}b_{qj}(w^{1/q})^{qj}\right|
=\left|g(w^{1/q})-\sum_{k=0}^{qn-1}b_{k}(w^{1/q})^{k}\right|\nonumber\\
&\le C_1A_1^{qn}M_{qn}|w^{1/q}|^{qn}.
\end{align}
We apply now the property (mg) of $\M$: it is straightforward to prove that there exists $A_0>0$ such that for all $n\in\N_0$ we have $M_{qn}\le A_0^{n}M_n^q$.
We may use this fact in~\eqref{equaAsympExpanfProofSurjGammaRational} and obtain that
\begin{equation*}
\left|f(w)-\sum_{j=0}^{n-1}a_jw^j\right|\le C_1(A_0 A_1^{q})^nM_{n}^q|w|^{n}.
\end{equation*}
So, $f\in\widetilde{\mathcal{A}}^u_{{\M}^{q}}(S_{p})$ and $\widetilde{\mathcal{B}}(f)=\widehat{f}$, what shows the surjectivity as intended.
\end{proof}

This result has several important consequences.

\begin{coro}\label{coroSurjIntervUniformGammaMnatural}
Let $\M$ be a strongly regular sequence with $\ga(\M)\in\mathbb{Q}$.
Then, $S_{\M}=\widetilde{S}^u_{\M}=(0,\ga(\M))$.
\end{coro}

\begin{proof}
By Theorem~\ref{th.SurjStronglyRegGammaMRational} and~\eqref{equaContentionSurjectIntervals}, we have $(0,\ga(\M))\subseteq S_{\M}\subseteq\widetilde{S}^u_{\M}$, while (iii)$\implies$(i) in Theorem~\ref{th.SurjStronglyRegGammaMRational} ensures that, $\ga(\M)$ being rational, it cannot be the case that $\ga(\M)\in \widetilde{S}^u_{\M}$, and so
$\widetilde{S}^u_{\M} \subseteq (0,\ga(\M))$.
\end{proof}

In the following $\mathbb{I}$ stands for the set of irrational numbers.

\begin{coro}\label{coro.SurjStrongRegGammaMReal}
Let $\M$ be a strongly regular sequence, and let $t\in\mathbb{R}$, $t>0$. Each assertion implies the following one:
\begin{itemize}
\item[(i)] $t<\ga(\M)$,
\item[(ii)] the Borel map $\widetilde{\mathcal{B}}:\mathcal{A}_{\M}(S_{t})\to \C[[z]]_\M$ is surjective,
\item[(iii)] the Borel map $\widetilde{\mathcal{B}}:\widetilde{\mathcal{A}}^u_{\M}(S_{t})\to \C[[z]]_\M$ is surjective,
\item[(iv)] the Borel map $\widetilde{\mathcal{B}}:\widetilde{\mathcal{A}}_{\M}(S_{t})\to \C[[z]]_\M$ is surjective,
\item[(v)] for every $\xi\in\mathbb{I}$ with $\xi<t$, the Borel map $\widetilde{\mathcal{B}}:\widetilde{\mathcal{A}}_{\M}(S_{\xi})\to \C[[z]]_\M$ is surjective,
\item[(vi)] $t\leq\ga(\M)$.
\end{itemize}
Hence, $(0,\ga(\M))\subseteq S_{\M} \subseteq \widetilde{S}^u_{\M} \subseteq \widetilde{S}_{\M} \subseteq (0,\ga(\M)]$.
\end{coro}
\begin{proof}
Only (v)$\implies$(vi) needs a short proof. For every $q\in\N$ we have that $\zeta=\xi q\notin\N$, we will show that  $\widetilde{\mathcal{B}}:\widetilde{\mathcal{A}}_{{\M}^{q}}(S_{\zeta})\to \C[[z]]_{\M^{q}}$ is surjective so, by  Theorem~\ref{th.SurjectivityNonUniformAsymp}.(i), we see that $\lfloor \zeta \rfloor <\ga(\M^{q})$. Then $\ga(\M)>  \lfloor \xi q \rfloor/q > \xi-1/q$. Since $q$ is arbitrary,  making $q$ tend to $\infty$ we deduce that $\xi\leq\ga(\M)$ for every irrational $\xi<t$, so $t\leq\ga(\M)$.\par

The proof of the surjectivity follows the same ramification argument used in (iv)$\implies$(i) of Theorem~\ref{th.SurjStronglyRegGammaMRational}, where the asymptotic relations obtained for bounded subsectors of $S_\xi$ are transformed into the analogous ones for the corresponding bounded subsectors of $S_\zeta$.
\end{proof}

\begin{rema}\label{rema.conjecture.Borelmap.surjective.SRS}
The situation for strongly regular sequences is summed up in Table~\ref{tableSurjectivityStrongReg}. The conjecture is that, at least for strongly regular sequences, one always has $\widetilde{S}_{\M}=(0,\ga(\M)]$ and $S_{\M}=\widetilde{S}^u_{\M}=(0,\ga(\M))$.
The main difference with the injectivity problem, in which the belonging of the value $\o(\M)$ to the injectivity interval depends on the convergence of a series, might lie in the fact that the value of $\ga(\M)$ completely characterized (snq) condition, that is, $\ga(\M)>0$ if and only if $\M$ has (snq), whereas for $\o(\M)$ we remember that if $\o(\M)>0$ then $\M$ is (nq), but if $\M$ is (nq) then only $\o(\M)\geq 0$ is known.
\end{rema}

\begin{table}[!htb]
\centering
\begin{tabular}{l|l|l|}
\cline{2-3}
 \rule{0pt}{1.1\normalbaselineskip} & $\ga(\M)\in \mathbb{Q}$  &      \multicolumn{1}{|c|}{$\ga(\M)\in\mathbb{I}$}       \\ \hline
\multicolumn{1}{|l}{$S_{\M}$} \rule{0pt}{1.1\normalbaselineskip}  &     \multicolumn{1}{|c|}{ $(0,\ga(\M))$} &            \multicolumn{1}{|c|}{ $(0,\ga(\M))$ or $(0,\ga(\M)]$} \\ \hline
\multicolumn{1}{|l}{$\widetilde{S}^u_{\M}$}  \rule{0pt}{1.1\normalbaselineskip} &     \multicolumn{1}{|c|}{ $(0,\ga(\M))$} &            \multicolumn{1}{|c|}{ $(0,\ga(\M))$ or $(0,\ga(\M)]$} \\  \hline
\multicolumn{1}{|l}{$\widetilde{S}_{\M}$}  \rule{0pt}{1.1\normalbaselineskip} &  \multicolumn{2}{|c|}{ $(0,\ga(\M))$ or $(0,\ga(\M)]$} \\ \hline
\end{tabular}
\caption{Surjectivity intervals for strongly regular sequences}
\label{tableSurjectivityStrongReg}
\end{table}

\begin{rema}\label{rema.interval.without.inject.surject}
A question which was open for some time is: Are $\ga(\M)$ and $\o(\M)$ always equal for strongly regular sequences? After some trial and error, a strongly regular sequence has been constructed with $\displaystyle\ga(\M)= 2<\o(\M) =5/2$ (see Example~2.2.26 in~\cite{PhDJimenez}, also Example 4.18 and Remark 4.19 in~\cite{JimenezSanzSchindlLCSNPO}).
In fact, given any pair of values $0<\ga<\o<\infty$ we are able to provide a strongly regular sequence $\M$ such that $\ga(\M)=\ga$ and $\o(\M)=\o$ (see Remark~2.2.27 in~\cite{PhDJimenez} and Subsection 4.3  in~\cite{JimenezSanzSchindlIndex}). This means that for opening $\a\pi$ with $\a$ in the interval $(\ga,\o)$, the Borel map is neither injective nor surjective and the corresponding injectivity and surjectivity intervals for this sequence are either $[\o,\infty)$ or $(\o,\infty)$ and $(0,\ga)$ or $(0,\ga]$, respectively.\par
\end{rema}

\subsection{Sequences admitting a nonzero proximate order}

In this final subsection, taking into account that the Borel map is never bijective, Theorem~\ref{theoNotBijectivity}, we will deduce more information regarding the surjectivity intervals. In order to be able to infer from that result whether or not $\ga(\M)$ belongs to $S_{\M}$ and $\widetilde{S}^u_{\M}$, strongly regularity is not enough and we need to
assume $\ga(\M)=\o(\M)$. Then,
\begin{itemize}
\item[(i)] If $\sum_{p=0}^{\infty} \left( m_{p}\right)^{-1/\o(\M)}=\infty$, we know that $\widetilde{I}^u_{\M}=I_{\M}=[\omega(\M),\infty)=[\ga(\M),\infty)$, and then
$$ S_{\M}= \widetilde{S}^u_{\M}=(0,\ga(\M)), \qquad (0,\ga(\M))\subseteq \widetilde{S}_{\M} \subseteq (0,\ga(\M)].$$
\item[(ii)] If $\sum_{p=0}^{\infty} \left(m_p\right)^{-1/\o(\M)}<\infty$ and $\sum_{p=0}^{\infty} \left((p+1)m_{p}\right)^{-1/(\o(\M)+1)}=\infty$, we know that $I_{\M}=[\ga(\M),\infty)$ and $\widetilde{I}^u_{\M}=(\ga(\M),\infty)$, and so
$$
S_{\M}= (0,\ga(\M)), \qquad (0,\ga(\M))\subseteq \widetilde{S}^u_{\M} \subseteq \widetilde{S}_{\M} \subseteq (0,\ga(\M)].
$$
\end{itemize}
Hence, the information we have for strongly regular sequences with $\ga(\M)=\o(\M)$  is summarized in the first two rows of Table~\ref{tableSurjectAdmitProxOrder}. Note that for nonuniform asymptotics this assumption does not produce any improvements and we will need to go one step further.\par

Our final result was given by the second author, Theorem 6.1 in \cite{SanzFlatProxOrder}, for strongly regular sequences $\M$ such that the function $d_{\M}$, defined by $d_{\M}(t):=\log(\o_{\M}(t))/\log(t)$, $t$ large enough, is a proximate order. For nonuniform asymptotics, he proved that $\widetilde{S}_{\M}=(0,\ga(\M)]$ employing the truncated Laplace transform technique, where the classical exponential kernel was replaced by a function which is constructed using proximate orders and Maergoiz's functions.
The weight sequences $\M$ for which $d_{\M}$ is a nonzero proximate order have been characterized in \cite[Theorem\ 3.6]{JimenezSanzSchindlLCSNPO}.
However, this property turned out not to be stable under equivalence, what motivated the study of a weaker condition which is indeed stable, as shown by the following statement.

\begin{theo}[\cite{JimenezSanzSchindlLCSNPO}, Theorem~4.14]\label{theoAdmissProxOrder}
 Let $\M$ be a weight sequence. The following
are equivalent:
 \begin{itemize}
\item[(a)] There exists a weight sequence $\L$
 and positive constants $A$ and $B$ such that $A^p L_p\leq M_p \leq B^p L_p$ and $d_{\L}(t)$ is a nonzero proximate order.
\item[(b)] $\M$ \emph{admits a nonzero proximate order} $\ro(t)$, i.e.,
there exist a nonzero proximate order $\ro(t)$ and constants $C$ and $D$ such that
$$ C\leq \log(t)\left(d_{\M}(t)- \ro(t)\right) \leq D, \qquad t\textrm{ large enough}.$$
\end{itemize}
\end{theo}
\begin{rema}\label{remacomentdderordenaprox}
\begin{itemize}
\item[(i)] The functions $\rho_{\a,\b}$ in the Example~\ref{examProxOrders} are admissible for the corresponding sequences $\M_{\a,\b}$ in the Example~\ref{exampleSequences}. This is useful even if, as it happens in this case, the functions $d_{\a,\b}(t):=\log(\o_{\M_{\a,\b}}(t))/\log(t)$ already are proximate orders, since $\rho_{\a,\b}$ are easier to handle and enjoy better regularity properties.
\item[(ii)] In the Gevrey case in particular, i.~e. for $\M_{\a}=(p!^{\a})_{\in\N_0}$,  the constant proximate order $\ro(r)\equiv 1/\a$ is admissible, and any $V\in MF(2\a,\ro(r))$ will provide us, by Theorem~\ref{teorconstrfuncplana}, with a flat function in the class $\widetilde{\mathcal{A}}_{\M_\a}(S_{\a})$. Since the choice $V(z)=z^{1/\a}$ is possible, we obtain the classical flat function $G(z)=\exp(-z^{-1/\a})$.
\end{itemize}
\end{rema}

As it is deduced from~\cite[Remark 4.11.(iii)]{SanzFlatProxOrder}, the construction in~\cite[Theorem\ 6.1]{SanzFlatProxOrder} is also available whenever  $\M$ is a weight sequence admitting a nonzero proximate order.  We recall that if $\M$ admits a nonzero proximate order then it is strongly regular and $\ga(\M)=\o(\M)\in(0,\infty)$ (see~\cite[Remark~4.15]{JimenezSanzSchindlLCSNPO}) but the converse does not hold~\cite[Example~4.16]{JimenezSanzSchindlLCSNPO}, so this is the most regular situation we will consider.

\begin{theo}[Generalized Borel--Ritt--Gevrey theorem]\label{th.BorelRittGevreySanz}
Let $\M$ be a weight sequence admitting a nonzero proximate order and $\ga>0$ be given. The following statements are equivalent:
\begin{itemize}
\item[(i)] $\gamma\le\omega(\M)=\ga(\M)$,
\item[(ii)] For every $\widehat{f}=\sum_{p\in\N_0} a_p z^p\in\C[[z]]_{\M}$ there exists a function $f\in\widetilde{\mathcal{A}}_{\bM}(S_{\gamma})$ such that $$f\sim_{\M}\widehat{f},$$
    i.e., $\widetilde{\mathcal{B}}(f)=\widehat{f}$. In other words, the Borel map $\widetilde{\mathcal{B}}:\widetilde{\mathcal{A}}_{\M}(S_{\gamma})\longrightarrow \C[[z]]_{\M}$ is surjective.
\end{itemize}
Hence, $\widetilde{S}_{\M}=(0,\ga(\M)]=(0,\o(\M)]$.
\end{theo}

 Table~\ref{tableSurjectAdmitProxOrder} gathers the information about surjectivity in case $\M$ admits a nonzero proximate order. For the sequence $\M_{\a,\b}=\big(p!^{\a}\prod_{m=0}^p\log^{\b}(e+m)\big)_{p\in\N_0}$, $\a>0$, $\b\in\R$, the information is summarized in Table~\ref{tableSurjecMAB}, note that the Gevrey case always belongs to the first column.

\begin{table}[h]
\centering
\begin{tabular}{l|l|l|l|l|}

\cline{2-5}
\rule{0pt}{1.05\normalbaselineskip} &     & \multicolumn{3}{|c|}{ $\ga(\M)\in\mathbb{I}$}   \\ \cline{3-5}
\rule{0pt}{1.3 \normalbaselineskip} & $\ga(\M)\in\mathbb{Q}$ & \scalebox{0.8}{$\displaystyle\sum_{p=0}^{\infty} \left(\frac{1}{m_{p}}\right)^{\frac{1}{\o(\M)}}=\infty$}
                       & \scalebox{0.8}{$\displaystyle\sum_{p=0}^{\infty} \left(\frac{1}{(p+1)m_{p}}\right)^{\frac{1}{\o(\M)+1}}=\infty$}
                       &  \multicolumn{1}{|l|}{\scalebox{0.8}{$\displaystyle\sum_{p=0}^{\infty} \left(\frac{1}{(p+1)m_{p}}\right)^{\frac{1}{\o(\M)+1}}<\infty$}} \\ \hline
	\multicolumn{1}{|l}{$S_{\M}$}   \rule{0pt}{1.1\normalbaselineskip}      &    \multicolumn{2}{|c}{$(0,\ga(\M))$}                   &    &                  \\ \cline{1-1}\cline{4-4}
   \multicolumn{1}{|l}{$\widetilde{S}^u_{\M}$} \rule{0pt}{1.1\normalbaselineskip}  &  \multicolumn{2}{|c|}{}                      & \multicolumn{2}{|c|}{ $(0,\ga(\M))$  or   $(0,\ga(\M)]$}                   \\ \hline
   \multicolumn{1}{|l}{$\widetilde{S}_{\M}$}   \rule{0pt}{1.1\normalbaselineskip}   &   \multicolumn{4}{|c|}{$(0,\ga(\M)]$}  \\ \hline
\end{tabular}
\caption{Surjectivity intervals for weight sequences admitting a nonzero proximate order.}
\label{tableSurjectAdmitProxOrder}
\end{table}

\begin{table}[h]
\centering

\begin{tabular}{l|l|l|l|}
\cline{2-4}

               & $\b\leq\a$
               & $ \a<\b\leq\a+1$
               & $\b>\a+1$ \\ \hline
	 \multicolumn{1}{|l|}{$S_{\M_{\a,\b}}$}  & $(0,\a)$                        & {$(0,\a)$}                   & {$(0,\a)$ or $(0,\a]$}                  \rule{0pt}{1.1\normalbaselineskip}         \\ \hline
  \multicolumn{1}{|l|}{$\widetilde{S}^u_{\M_{\a,\b}}$}& $(0,\a)$                       & {$(0,\a)$ or $(0,\a]$}    & {$(0,\a)$ or $(0,\a]$}          \rule{0pt}{1.1\normalbaselineskip}               \\ \hline
  \multicolumn{1}{|l|}{$\widetilde{S}_{\M_{\a,\b}}$}  &  $(0,\a]$ &  $(0,\a]$ &   $(0,\a]$ \rule{0pt}{1.1\normalbaselineskip}\\ \hline
\end{tabular}
\caption{Surjectivity intervals for  the sequences $\M_{\a,\b}$, $\a>0$, $\b\in\R$.}
\label{tableSurjecMAB}
\end{table}

\noindent\textbf{Acknowledgements}: The first two authors are partially supported by the Spanish Ministry of Economy, Industry and Competitiveness under the project MTM2016-77642-C2-1-P. The first author is partially supported by the University of Valladolid through a Predoctoral Fellowship (2013 call) co-sponsored by the Banco de Santander. The third author is supported by FWF-Project J~3948-N35, as a part of which he is an external researcher at the Universidad de Valladolid (Spain) for the period October 2016 - September 2018.\par

\vskip.5cm

\noindent\textbf{Affiliation}:\\
J.~Jim\'{e}nez-Garrido, J.~Sanz:\\
Departamento de \'Algebra, An\'alisis Matem\'atico, Geometr{\'\i}a y Topolog{\'\i}a\\
Universidad de Va\-lla\-do\-lid\\
Facultad de Ciencias, Paseo de Bel\'en 7, 47011 Valladolid, Spain.\\
Instituto de Investigaci\'on en Matem\'aticas IMUVA\\
E-mails: jjjimenez@am.uva.es (J.~Jim\'{e}nez-Garrido), jsanzg@am.uva.es (J. Sanz).
\\
\vskip.5cm
\noindent G.~Schindl:\\
Departamento de \'Algebra, An\'alisis Matem\'atico, Geometr{\'\i}a y Topolog{\'\i}a\\
Universidad de Va\-lla\-do\-lid\\
Facultad de Ciencias, Paseo de Bel\'en 7, 47011 Valladolid, Spain.\\
E-mail: gerhard.schindl@univie.ac.at.
\end{document}